\documentclass[12pt]{article}
\usepackage{e-jc}
\usepackage{float}
\usepackage{upgreek}
\usepackage{booktabs, enumerate,float}
\usepackage[mathcal]{euscript}
\usepackage{amsthm,amsmath,amssymb}
\usepackage{tikz}
\usetikzlibrary{arrows,matrix,positioning}
\usepackage{float,multirow}
\usepackage{graphicx,tipa}
\newcommand{\arc}[1]{{%
\setbox9=\hbox{$#1$}%
\ooalign{\resizebox{\wd9}{\height}{\texttoptiebar{\phantom{A}}}\cr$#1$}}}
\definecolor{webgreen}{rgb}{0,.5,0}
\definecolor{webbrown}{rgb}{.6,0,0}
\usepackage[colorlinks=true,
linkcolor=webgreen,
filecolor=webbrown,
citecolor=webgreen]{hyperref}

\newcommand{\arxiv}[1]{\href{http://arxiv.org/abs/#1}{\texttt{arXiv:#1}}}
\newcommand{\seqnum}[1]{\href{http://oeis.org/#1}{\underline{#1}}}
\newcommand{\Figs}[1]{\hyperref[#1]{Figure~\ref*{#1}}}
\newcommand{\Tabs}[1]{\hyperref[#1]{Table~\ref*{#1}}}
\everymath{\displaystyle}
\theoremstyle{plain}
\newtheorem{theorem}{Theorem}
\newtheorem{lemma}[theorem]{Lemma}
\newtheorem{corollary}[theorem]{Corollary}
\newtheorem{proposition}[theorem]{Proposition}

\theoremstyle{definition}
\newtheorem{definition}[theorem]{Definition}
\newtheorem{example}[theorem]{Example}

\newtheorem{notation}[theorem]{Notation}

\theoremstyle{remark}
\newtheorem{remark}[theorem]{Remark}

\title{\bf Enumerating the states of the twist knot}

\author{Franck Ramaharo\\
\small D\'epartement de Math\'ematiques et Informatique\\[-0.8ex]
\small Universit\'e d'Antananarivo\\[-0.8ex] 
\small 101 Antananarivo, Madagascar\\
\small\href{mailto:franck.ramaharo@gmail.com}{\tt franck.ramaharo@gmail.com}\\
}

\date{\small\today\\
\small 2010 Mathematics Subject Classifications: 57M25; 94B25.}

\begin{document}

\maketitle

\begin{abstract}
We enumerate the state diagrams of the twist knot shadow which consist of the disjoint union of two trivial knots. The result coincides with the maximal number of regions into which the plane is divided by a given number of circles. We then establish a bijection between the state enumeration and this particular partition of the plane by means of binary words.

\bigskip\noindent \textbf{Keywords:} twist knot, knot shadow, state diagram, planar arrangement of circles.
\end{abstract}

\section{Introduction}

Knot theory defines a mathematical knot as a closed curve in three-dimensional space that does not intersect itself. We usually describe knots by drawing the so-called \textit{knot diagram}, a generic projection of the knot to the plane or the sphere with finitely many double points called \textit{crossings}. It is indicated at each such point which strand crosses over and which one crosses under usually by erasing part of the lower strand. Additionally, we might also continuously deform a knot diagram such that we obtain another representation which is \textit{planar isotopic} to the former diagram, and the crossings remaining unaltered, i.e., no introduction of new crossings and no removal of the existing ones \cite[p.\ 12]{Adams}. The \textit{shadow diagram}, or shortly the \textit{shadow} of the knot is the regular projection onto the plane that omits the crossing information \cite{DD}. If we let $ \mathcal{S} $ denote a shadow, then we call a component of $ \mathbb{R}^2\setminus \mathcal{S} $ a \textit{region} of $\mathcal{S} $. We say that two regions are adjacent if the contours of the two regions have at least an arc in common, and we say that two regions are opposite if their contour have exactly a double point in common. A \textit{checkerboard coloring} \cite{HK} of a shadow is a coloring of every region of $ \mathcal{S} $ to be black or white so that black regions are only adjacent to white regions, and conversely, white regions are only adjacent to black ones. Once colored, we label each region: the unbounded one is labeled as a $ A$ and so are the regions of the same color. The adjacent regions to an \textit{$ A $-region} are then labeled as $ B $ (see \Figs{Fig:FigureEightDiagram}). 

\begin{figure}
\centering
\includegraphics[width=.5\linewidth]{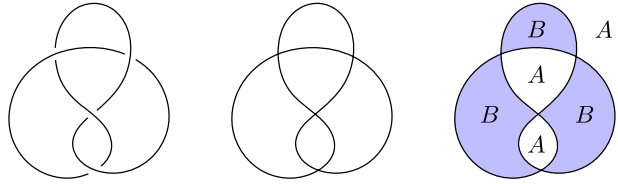}
\caption{The diagram of the \textit{figure-eight knot} and its shadow diagram as well as the corresponding labels and \textit{checkerboard coloring}.}
\label{Fig:FigureEightDiagram}
\end{figure}

A common operation to perform on a shadow’s crossing is to \textit{split} it in one way or the other: we remove the crossing and glue the arcs so that either of the two opposite regions are merged into one \cite[p.\ 27]{Kauffman1}. The \textit{$ A$-split} joins the \textit{$ A $-regions} while the \textit{$ B $-split} joins the $ B $-regions as illustrated in \Figs{Fig:StateOfCrossingFigureEight}.

\begin{figure}
\centering		
\includegraphics[width=0.5\linewidth]{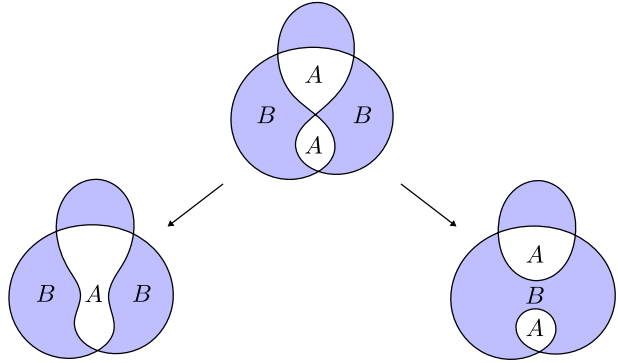}
\caption{The $ A $-split and the $ B $-split.}
\label{Fig:StateOfCrossingFigureEight}		
\end{figure}

Thus, any $ n $-crossing knot diagram can be decomposed into $ 2^n $ final descendants without crossing, which are called \textit{states} of the diagram \cite[p.\ 71]{Manturov}. A \textit{state} is a collection of separated component called \textit{trivial knot}. We also call such knot \textit{unknot} whose diagram is a simple closed loop. We shall refer to \textit{$ k $-state} a state which consists of $ k $ components.

\begin{example}
We repeat the operation we performed to the previous figure-eight knot, and we obtain a complete family of state diagrams of trivial knots (see \Figs{Fig:StateDiagramsFigureEight}).
\begin{figure}[H]
\centering		
\includegraphics[width=.9\linewidth]{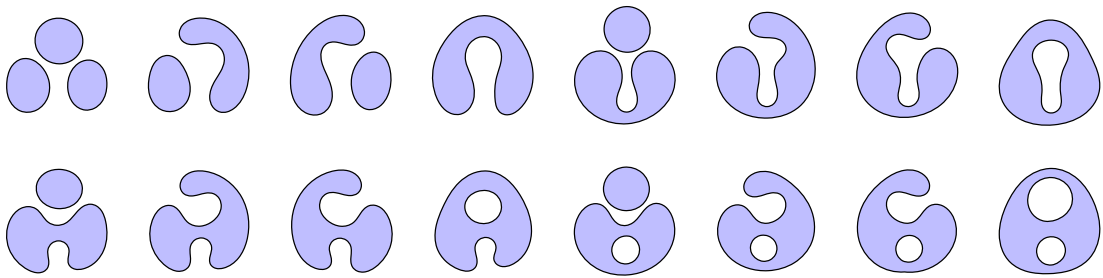}
\caption{The state diagrams of the figure-eight knot.}
\label{Fig:StateDiagramsFigureEight}		
\end{figure}
\end{example}

Throughout this paper, we only consider shadow diagrams, and for the sake of simplicity, we abusively refer to the shadow diagram by either of the terms knot, knot diagram, or diagram.

\begin{definition}
Let $ S $ denote a state of a $ n $-crossing diagram $ D $, and let $ |S| $ denote the number of components of the state $ S $. Also, let $ d_{n,k} $ enumerates the $ k$-state diagrams. We define the generating polynomial by the formula
\begin{equation*}
D(x):=\sum_{S} x^{|S|}=\sum_{k\geq 0}d_{n,k}x^k,
\end{equation*} 
where the summation is taken over all states for $ D $.
\end{definition}

Hence, the generating polynomial of the figure-eight diagram is given by
\begin{equation*}
D(x)=3x^3+8x^2+5x.
\end{equation*}

Associating a knot operation with a polynomial is one of the most usual ways of classifying and enumerating knots. The so-called \textit{knot polynomials} \cite{JR,Kauffman} are invariant polynomials which help to decide whether two different looking pictures in fact represent the same knot. In this paper, we follow the combinatoric approach of Kauffman by using a simplified version of the bracket polynomial \cite{Kauffman}. We express our polynomial only by means of the number components of the states without including any further topological information. This makes sense since we only focus on the shadow diagrams.

The rest of this paper is organized as follows. In \hyperref[Sec:TwsitKnot]{section~\ref*{Sec:TwsitKnot}}, we establish the generating polynomial of the twist knot of $ n $ crossings as well as some related knots. We pay a particular attention to the cardinality of the class of the $ 2 $-states as it turns out to be the maximal number of regions into which the plane is divided by $ n+1 $ circles in general arrangement. In \hyperref[Sec:Rosette]{section~\ref*{Sec:Rosette}}, we establish an arrangement of circles that meets the previously mentioned criterion. The regions of the corresponding partition of the plane are next encoded with binary words in \hyperref[Sec:RosetteEncoding]{section~\ref*{Sec:RosetteEncoding}}. We also this encoding scheme to the states of the twist knot in \hyperref[Sec:StatesEncoding]{section~\ref*{Sec:StatesEncoding}}. Finally we establish a bijection between the partition of the plane and the class of the $ 2 $-states in \hyperref[Sec:Bijection]{section~\ref*{Sec:Bijection}}.

\section{Twist knot}\label{Sec:TwsitKnot}

First of all, let us focus on the following three family of knots.
\begin{definition}
Let $ n $ be a natural number. A \textit{twist loop} is a knot obtained by repeatedly twisting a closed loop. We call a twist loop of $ n $ half-twists a \textit{$ n $-twist loop}, and we refer to such knot as $ T_n $. For instance, the $ 0 $-twist loop and the $ 6 $-twist loop are illustrated in \Figs{Fig:6TwsitLoop}.
\begin{figure}[H]
\centering
\includegraphics[width=0.4\linewidth]{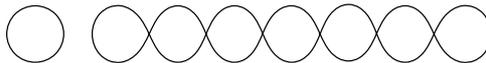}
\caption{The $ 0 $-twist loop (unknot) and the $ 6 $-twist loop.}
\label{Fig:6TwsitLoop}
\end{figure}
\end{definition}

We introduce the second simplest family of knot by applying a slight surgery to the twist loop.
\begin{definition}
We construct a \textit{foil knot} or shortly a \textit{foil}  \cite{RR} of $ n$ half-twists by removing a little arc from the right-end and the left-end of the $ n $-wist loop, then connecting the ends in pairs by arcs that do not cross each other as in \Figs{Fig:6TwistLoopFoil}. We call a knot of such kind a \textit{$ n $-foil} and we refer to it as $ F_n $.
\begin{figure}[H]
\centering
\includegraphics[width=.6\linewidth]{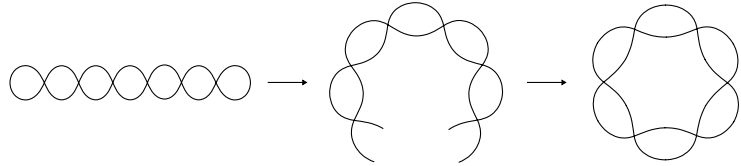}
\caption{Constructing a $ 6 $-foil knot from a $ 6 $-twist loop}
\label{Fig:6TwistLoopFoil}
\end{figure}
\end{definition}

In the same idea, we define the twist knot as follows.
\begin{definition}
We let $ \uptau_n $ denote a \textit{twist knot} \cite{JH} of $ n$ half-twists which is a knot obtained by linking the ends of a $ n $-twist loop together. Therefore, we may consider the twist knot as two-parts knot: the \textit{twist part} and the \textit{link part}. The former consists of $ n $ half-twists and the latter consists of two crossings (see \Figs{Fig:6TwistKnot}). 
\begin{figure}[H]
\centering
\includegraphics[width=.5\linewidth]{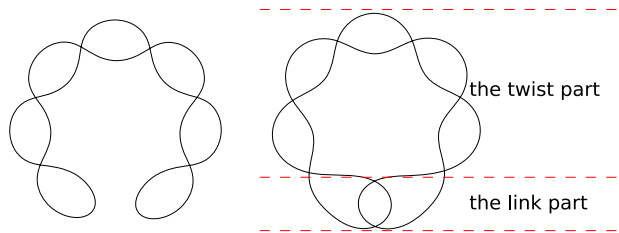}
\caption{The two parts of the twist knot: the twist part and the link part.}
\label{Fig:6TwistKnot}
\end{figure}
\end{definition}

\begin{remark}
Following the construction of the foil knot and the twist knot, we have the corresponding representation of the $ 0 $-foil $ F_0 $ and the $ 0 $-twist knot $ \uptau_0 $ in \Figs{Fig:0FoilTwistKnot}. Notice that the representation of the $ 0 $-twist knot always has 2 crossings in the link part.
\begin{figure}[H]
\centering
\includegraphics[width=.35\linewidth]{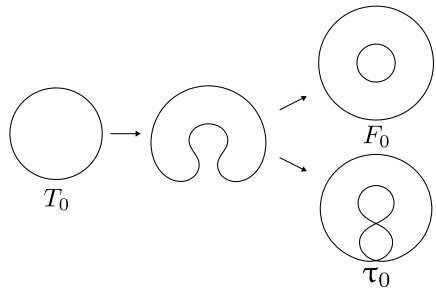}
\caption{Construction of the $ 0 $-foil and the $ 0 $-twist knot from the $ 0 $-twist loop.}
\label{Fig:0FoilTwistKnot}
\end{figure}
\end{remark}

We have the following immediate result.
\begin{corollary}
The generating polynomials of the $ 0 $-twist loop and the $ 0 $-foil knot are respectively given by
\begin{equation*}
T_0(x)=x
\end{equation*}
and
\begin{equation*}
F_0(x)=x^2.
\end{equation*}
\end{corollary}

In order to compute the generating polynomial $ \uptau_n(x) $ of the $ n $-twist knot, let us first focus on how the split operation affects the link part. The results are four intermediate state diagrams as illustrated in \Figs{Fig:StateOfLinkPart}.
\begin{figure}[H]
\centering
\includegraphics[width=0.4\linewidth]{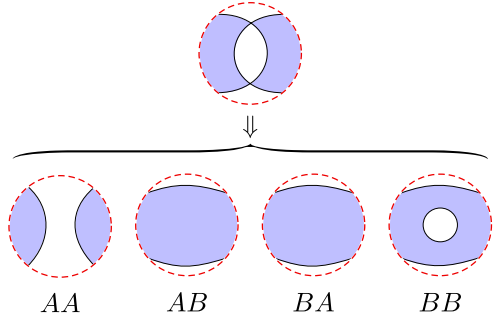}
\caption{The states of the link part.}
\label{Fig:StateOfLinkPart}
\end{figure}
The words $ AA $, $ AB $, $ BA $ and $ BB $ represent the splits sequence we have applied to the link part. As an immediate application, we give the generating polynomial of the $ 0 $-twist knot.

\begin{corollary}
The generating polynomial of the $ 0 $-twist knot is
\begin{equation*}
\uptau_0(x)=x^3+2x^2+x.
\end{equation*}
\end{corollary}

\begin{proof}
The states of the $ 0 $-twist knot are illustrated in \Figs{Fig:0TwsitKnot}. The associated generating polynomial is then given by $ 	\uptau_0(x)=x+x^2+x^2+x^3$.
\begin{figure}[H]
\centering
\includegraphics[width=0.5\linewidth]{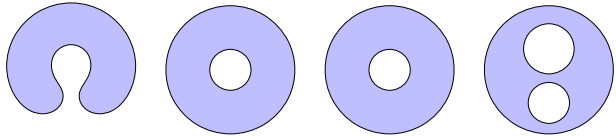}
\caption{The states of the $ 0 $-twist knot}
\label{Fig:0TwsitKnot}
\end{figure}

\end{proof}

\begin{proposition}
The generating polynomial of the $ n $-twist knot is given by
\begin{equation}\label{eq:uptaunx}
\uptau_n(x)=T_n(x)+(x+2)F_n(x),\ n\geq 1,
\end{equation}
where $ T_n(x) $ and $ F_n(x) $ respectively denote the generating polynomial of the $ n $-twist loop and the $ n $-foil knot.
\end{proposition}

\begin{proof}
Given a $ n $-twist knot, we split the link part and we obtain four diagrams: a $ n $-twist loop, two $ n $-foil knots and the disjoint union of the unknot and a $ n $-foil knot (\Figs{Fig:StateOfCrossingTwistKnot}). 
\begin{figure}[H]
\centering
\includegraphics[width=0.7\linewidth]{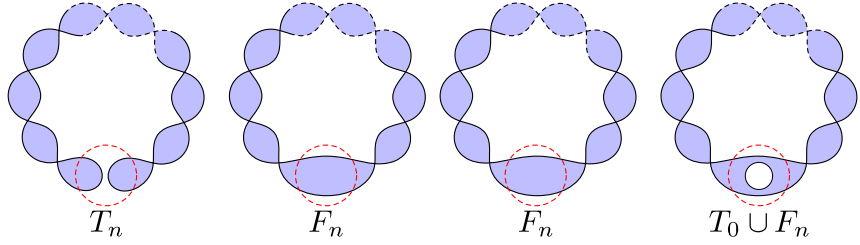}
\caption{The four intermediate states of the the twist knot.}
\label{Fig:StateOfCrossingTwistKnot}
\end{figure}
Hence we write
\begin{equation*}
\uptau_n(x)=T_n(x)+2F_n(x)+xF_n(x).
\end{equation*}
\end{proof}

We can solve the expression of $ \uptau_n(x) $ for the closed form following $ T_n(x) $ and $ F_n(x) $.
\begin{proposition}
The generating polynomial of the $ n $-twist loop is given by
\begin{equation}\label{eq:twistloop}
T_n(x)=x(x+1)^n,\ n\geq 0.
\end{equation}
\end{proposition}

\begin{proof}
When $ n=0 $, we recover $ T_0(x)=x $. Let $ n\geq 1 $, we split the leftmost crossing so that the resulting diagram is either a disjoint union of the unknot and a $ (n-1) $-twist loop, or uniquely a $ (n-1) $-twist loop (see \Figs{Fig:StateOfCrossingTwistLoop}).
\begin{figure}[H]
\centering
\includegraphics[width=0.8\linewidth]{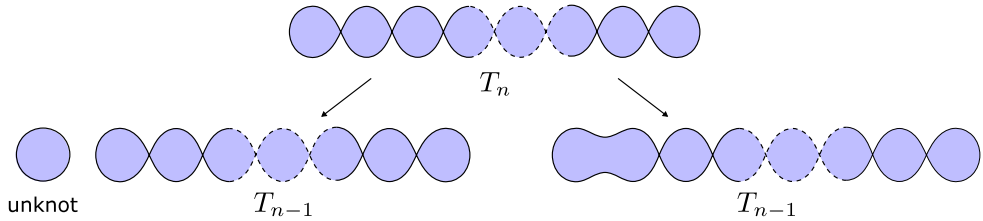}
\caption{The states of one crossing of the $ n$-twist loop.}
\label{Fig:StateOfCrossingTwistLoop}
\end{figure}
The corresponding generating polynomial is therefore
\begin{equation}\label{eq:Tnx}
T_n(x)=xT_n(x)+T_{n-1}(x).
\end{equation}
Taking into consideration the expression $ T_0(x) =x$, we obtain
\begin{equation*}
T_n(x)=x(x+1)^n.
\end{equation*}
\end{proof}

Writing $ T_n(x):=\sum_{k\geq 0}t_{n,k}x^k $ and identifying the coefficients of the generating polynomial, formulas \eqref{eq:twistloop} and \eqref{eq:Tnx} allow us to define the following recurrence:
\begin{equation}\label{eq:tnk}
\begin{cases}
t_{n,0}=0,\ t_{n,1}=t_{n,n+1}=1, & n\geq 0;\\
t_{n,k}=t_{n-1,k}+t_{n-1,k-1},	&1< k\leq n+1;
\end{cases}
\end{equation}
which gives Table~\ref{tab:twistloop} for $ 0\leq n\leq 6 $ and $ 0\leq k\leq 7 $.

\begin{table}[H]
\centering
$
\begin{array}{c|rrrrrrrr}
n\ \backslash\ k	&0	&1	&2	&3	&4	&5	&6	&7\\
\midrule
0	&0	&1	&	&	&	&	&	&\\
1	&0	&1	&1	&	&	&	&	&\\
2	&0	&1	&2	&1	&	&	&	&\\
3	&0	&1	&3	&3	&1	&	&	&\\
4	&0	&1	&4	&6	&4	&1	&	&\\
5	&0	&1	&5	&10	&10	&5	&1	&\\
6	&0	&1	&6	&15	&20	&15	&6	&1\\
\end{array}
$
\caption{Values of $ t_{n,k} $ for $ 0\leq n\leq 6 $ and $ 0\leq k\leq 7 $.}
\label{tab:twistloop}
\end{table}

We notice that the values in \Tabs{tab:twistloop} represent a horizontal-shifted binomial coefficients table with $ t_{n,k}=\binom{n}{k-1},\ 1\leq k\leq n+1$. Besides, we should mention that the constraint $ t_{n,0}=0$, $n\geq 0 $, actually holds for any knot since a state is at least composed by one component. 

\begin{proposition}
The $ n $-foil knot has the following generating polynomial:
\begin{equation}\label{eq:genepolyFn}
F_n(x)=(x+1)^n+x^2-1,\ n\geq 0.
\end{equation}
\end{proposition}

\begin{proof}
When $ n=0 $, we verify $ F_0(x)=x^2 $. Let $ n\geq 1 $, we split a crossing of the $ n $-foil as illustrated in \Figs{Fig:StateOfCrossingFoil}, and we obtain either a $ (n-1) $-twist loop or a $ (n-1) $-foil knot.
\begin{figure}[H]
\centering
\includegraphics[width=0.5\linewidth]{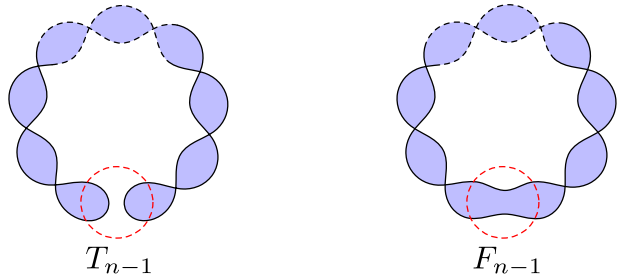}
\caption{The states of one crossing of the $ n$-foil.}
\label{Fig:StateOfCrossingFoil}
\end{figure}
We then write 
\begin{equation}\label{eq:Fnx}
F_n(x)=T_{n-1}(x)+F_{n-1}(x).
\end{equation}
Since 	$ T_n(x)=x(x+1)^n $ and $ F_0(x)=x^2 $, we get
\begin{align*}
F_n(x)&=\displaystyle{\sum_{k=0}^{n-1}T_k(x)+F_0(x)}\\
&=(x+1)^n-1+x^2.
\end{align*}
\end{proof}

As previously, if we write $ F_n(x):=\sum_{k\geq 0}f_{n,k}x^k$, then we obtain the following relation by combining \eqref{eq:genepolyFn} and \eqref{eq:Fnx}:
\begin{equation}\label{eq:fnk}
\begin{cases}
f_{0,1}=0,\ f_{0,2}=f_{1,2}=1;\\ 
f_{n,0}=0, & n\geq 0;\\
f_{n,k}=f_{n-1,k}+t_{n-1,k},	&0< k\leq n.
\end{cases}
\end{equation}
The values of $ f_ {n,k}$ are then arranged in \Tabs{tab:gpfoil} for $ 0\leq n\leq k\leq 12 $.
\begin{table}[H]
\centering
$
\begin{array}{c|rrrrrrrrrrrrr}
n\ \backslash\ k	&0	&1	&2	&3	&4	&5	&6	&7	&8	&9	&10	&11	&12\\
\midrule
0	&0	&0	&1	&	&	&	&	&	&	&	&	&	&\\
1	&0	&1	&1	&	&	&	&	&	&	&	&	&	&\\
2	&0	&2	&2	&	&	&	&	&	&	&	&	&	&\\
3	&0	&3	&4	&1	&	&	&	&	&	&	&	&	&\\
4	&0	&4	&7	&4	&1	&	&	&	&	&	&	&	&\\
5	&0	&5	&11	&10	&5	&1	&	&	&	&	&	&	&\\
6	&0	&6	&16	&20	&15	&6	&1	&	&	&	&	&	&\\
7	&0	&7	&22	&35	&35	&21	&7	&1	&	&	&	&	&\\
8	&0	&8	&29	&56	&70	&56	&28	&8	&1	&	&	&	&\\
9	&0	&9	&37	&84	&126	&126	&84	&36	&9	&1	&	&	&\\
10	&0	&10	&46	&120	&210	&252	&210	&120	&45	&10	&1	&	&\\
11	&0	&11	&56	&165	&330	&462	&462	&330	&165	&55	&11	&1	&\\
12	&0	&12	&67	&220	&495	&792	&924	&792	&465	&220	&66	&12	&1
\end{array}
$
\label{tab:gpfoil}
\caption{Values of $ f_{n,k} $ for $ 0\leq n\leq k\leq 12 $.}
\end{table}

We browse the columns in Table~\ref{tab:gpfoil} and identify the corresponding OEIS \cite{Sloane} entries.
\begin{itemize}
\item We recognize the common constraint $ f_{n,0}=0$, $ n\geq 0 $.

\item We have $ f_{n,1}=\seqnum{A001477}(n)=n $, $ n\geq 0 $. It represents the number of $ 1 $-states that is obtained by applying an $ A $-split at a chosen crossing, and applying a $ B $-split at the remaining $ n-1 $ crossings. 
We justify the value $ f_{0,1}=0$ since the $ 0 $-foil already consists of two components.

\item In the third column we have $ f_{n,0}=1 $ and $ f_{n,2}=\seqnum{A000124}(n)=\binom{n}{2}+1$, $ n\geq 1 $. Here, $ (f_{n,2})_{n\geq1} $ is the \textit{lazy caterer's sequence} \cite[p.\ 5]{GKP} which describes the maximum number of pieces of a disk that can be made with a given number of straight cuts.
\item Finally, when $ k\geq3 $, the remaining columns represent the usual binomial coefficients, i.e., $ f_{n,k}=\binom{n}{k} $, $ n\geq k $.
\end{itemize}

We can now give the closed form of the generating polynomial of the twist knot.
\begin{corollary}
The $ n $-twist knot has the following generating polynomial
\begin{equation}\label{eq:closedformtwistknot}
\uptau_n(x)=2(1+x)^{n+1}+x^3+2x^2-x-2,\ n\geq 0.
\end{equation}
\end{corollary}

By formulas \eqref{eq:uptaunx} and \eqref{eq:closedformtwistknot}, we deduce the recurrence that defines the coefficients of the polynomial $ \uptau_n(x) =\sum_{k\geq 0}\tau_{n,k}x^k $, namely

\begin{equation}\label{eq:taunk}
\begin{cases}
\tau_{n,0}=0,\ \tau_{0,3}=\tau_{1,3}=1, 				& n\geq 0;\\
\tau_{n,k}=f_{n,k-1}+2f_{n,k}+t_{n,k},	&0< k\leq n+1.
\end{cases}
\end{equation}
Next, let us arrange the coefficients in \eqref{eq:taunk} in \Tabs{tab:twistknot} for $ 0\leq n\leq k\leq 12 $, and identify the corresponding columns. 
\begin{table}[H]
\centering
$
\begin{array}{c|rrrrrrrrrrrrr}
n\ \backslash\ k	&0	&1	&2	&3	&4	&5	&6	&7	&8	&9	&10	&11	&12\\
\midrule
0	&0	&1	&2	&1	&	&	&	&	&	&	&	&	&\\
1	&0	&3	&4	&1	&	&	&	&	&	&	&	&	&\\
2	&0	&5	&8	&3	&	&	&	&	&	&	&	&	&\\
3	&0	&7	&14	&9	&2	&	&	&	&	&	&	&	&\\
4	&0	&9	&22	&21	&10	&2	&	&	&	&	&	&	&\\
5	&0	&11	&32	&41	&30	&12	&2	&	&	&	&	&	&\\
6	&0	&13	&44	&71	&70	&42	&14	&2	&	&	&	&	&\\
7	&0	&15	&58	&113	&140	&112	&56	&16	&2	&	&	&	&\\
8	&0	&17	&74	&169	&252	&252	&168	&72	&18	&2	&	&	&\\
90	&0	&19	&92	&241	&420	&504	&420	&240	&90	&20	&2	&	&\\
10	&0	&21	&112	&331	&660	&924	&924	&660	&330	&110	&22	&2	&\\
11	&0	&23	&134	&441	&990	&1584	&1848	&1584	&990	&440	&132	&24	&2
\end{array}
$
\caption{Array value of the coefficients of $ \uptau_n(x) $, $ 0\leq n \leq k\leq 11 $.}
\label{tab:twistknot}
\end{table}

\begin{itemize}
\item Again, $ \tau_{n,0}=0 $, $ n\geq0 $, which is the common constraint.
\item We have $ \tau_{n,1}=\seqnum{A005408}(n)=2n+1 $, $ n\geq 0 $.
\item For $ n\geq 0 $, we find $ \tau_{n,2}=\seqnum{A014206}(n)=n^2+n+2$, which is the maximal number of regions into which the plane is divided by $ n+1 $ circles in general arrangement. We give further details in the next section.
\item When $ k=3 $, the corresponding column is defined by the initial value $\tau_{0,3}=1 $ and the $n $-th term $ \tau_{n,3}=\seqnum{A064999}(n)=\frac{1}{3}\big(n^3-n+3\big)$, $ n\geq 1 $. 
\item Finally, when $ k\geq 4 $, we have $ \tau_{n,k}=2\binom{n+1}{k}$, $ n\geq k-1 $. This is the usual binomial coefficients, horizontally shifted and doubled.
\end{itemize}

\begin{remark}	
From \Tabs{tab:twistloop} and \Tabs{tab:gpfoil}, we read 
\begin{equation}\label{eq:T1F0}
T_1(x)=F_1(x)=x^2+x,
\end{equation}
and from \Tabs{tab:twistloop} and \Tabs{tab:twistknot},
\begin{equation}\label{eq:T2tau0}
T_2(x)=\uptau_0(x)=x^3+2x^2+x
\end{equation}
Also, checking \Tabs{tab:gpfoil} and \Tabs{tab:twistknot}, we have
\begin{equation}\label{eq:F3tau1}
F_3(x)=\uptau_1(x)=x^3+4x^2+3x.
\end{equation}
We can explain these equalities by introducing the following definition.
\end{remark}	

\begin{definition}[Denton and Doyle \cite{DD}] Let us draw the shadow diagram on a sphere. When we have a loop on the outside edge of the diagram, we can redraw this loop around the other side of the diagram by pulling the entire loop around the far side of the sphere without affecting the constraints on any of the already existing crossings (see \Figs{Fig:0SMove}). We call the move a \textit{type 0 move on the sphere}, denoted $ 0S^2 $. 
\begin{figure}
\centering
\includegraphics[width=.7\linewidth]{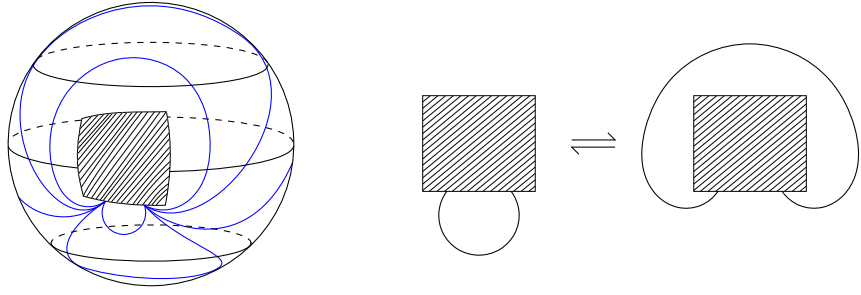}
\caption{Denton--Doyle type $ 0 $ move on the sphere.}
\label{Fig:0SMove}
\end{figure}
\end{definition}

Since this particular move does not remove nor create a crossing, the number of region remains intact. Knots which are planar isotopic or related by a sequence of $ 0S^2 $ moves then have the same state diagrams (modulo some rearrangements), and thus, the same generating polynomial. The equalities \eqref{eq:T1F0}, \eqref{eq:T2tau0} and \eqref{eq:F3tau1} result from this property (see \Figs{Fig:EquivalentKnot}).
\begin{figure}[H]
\centering
\includegraphics[width=\linewidth]{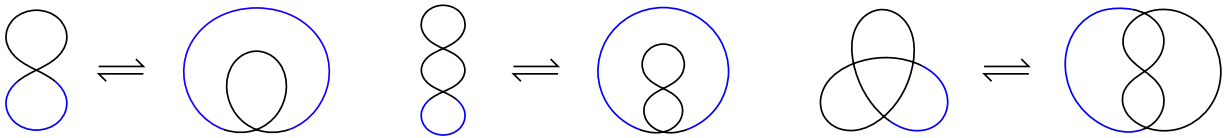}
\caption{Knots $ T_1\equiv F_1$, $ T_2\equiv \uptau_0$ and $ T_3\equiv \uptau_1$ under $ 0S^2 $ move.}
\label{Fig:EquivalentKnot}
\end{figure}

\section{A particular arrangement of circles}\label{Sec:Rosette}
Let $ \mathcal{U} $ denote an oriented unit circle of center $ O $ in the Euclidean space. We shall refer to this circle as \textit{reference circle}. Throughout this paper, we let $ \mathfrak{C} $ denote a finite family of circles with radius $ r>1 $ and centered at the border of $ \mathcal{U} $. We additionally impose that circles in $ \mathfrak{C} $ are congruent and non-concentric. We let $ \mathcal{C}_n $, $ n>0 $, denote a circle in $ \mathfrak{C} $, of center $ C_n $, and we refer to the interior of that circle as $ \mathcal{D}_n$, i.e., the closed disk bounded by that circle.

Let $ n>0 $, we say that $ n $ circles are in general arrangement when
\begin{itemize}
\item they intersect pairwise, that is, if no two of them are tangent and none of them lies entirely within or outside of another one;
\item no three of them share a common point.
\end{itemize}

Circles in general arrangement divide the plane into a maximum number of regions that is given by
\begin{equation}\label{eq:planedivision}
P(n)=n^2-n+2.
\end{equation}
For $ n=1,2,3,\ldots $ formula \eqref{eq:planedivision} gives $ 2,4,8,14,22,32,44,58,\ldots $ (sequence \seqnum{A0142016}).

The aim of the present section is to show that the circles in $ \mathfrak{C} $ are in general arrangement. We begin with the following classic result.

\begin{theorem}[{\cite[p.\ 96]{CR}}]
The perpendicular bisector of a chord passes through the center of the circle.
\end{theorem}

\begin{proof}
An elegant proof is to consider the chord of the circle to be the side of an arbitrary inscribed triangle whose three perpendicular bisectors must intersect in one point, the center of the circle.
\end{proof}

\begin{proposition}\label{Prop:int3}
Given three circles in $ \mathfrak{C} $, exactly one of the intersection points of the two circles lies at the interior of the third circle.
\end{proposition}
\begin{proof}
Let $ \mathcal{C}_1 $ and $ \mathcal{C}_2 $ denote the two circles, and let $ I $ and $ I'$ be their intersection points. The circles are the mirror-image symmetry of each other and whose line of symmetry is $ \left(II'\right) $. At this stage, the points $ I,O,I' $ are, in this order, collinear since the line $ \left(II'\right) $ is also the perpendicular bisector of the chord $ [C_1C_2] $ (see \Figs{Fig:3CirclesIntersection}). 
\begin{figure}
\centering
\includegraphics[width=0.7\linewidth]{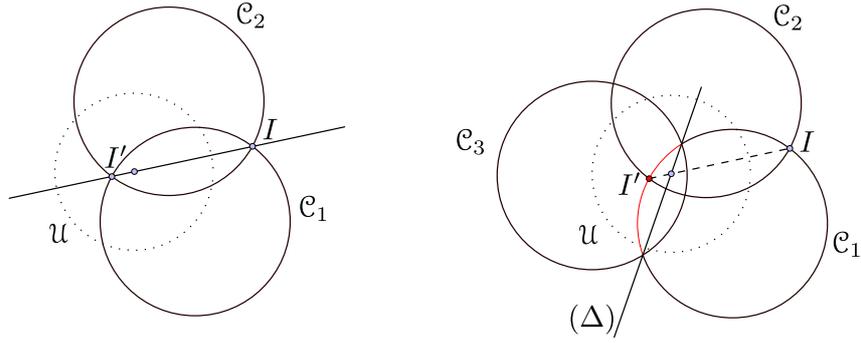}
\caption{The circle $ \mathcal{C}_3 $ contains one of the intersection points of the circles $ \mathcal{C}_1 $ an $ \mathcal{C}_2 $.}
\label{Fig:3CirclesIntersection}
\end{figure}
The same scheme applies when we add a third circle $ \mathcal{C}_3 $. Let $ (\Delta) $ denote the line of symmetry of the circles $ \mathcal{C}_1 $ and $ \mathcal{C}_3 $. The line $(\Delta) $ passes through $ O $ so that the points $ I $ and $ I' $ are in either side of $ (\Delta) $. Now, owning to the property of the reflection with respect to the line $ (\Delta) $, only one of $ I $ or $ I' $ is contained in $ \mathcal{C}_3 $.
\end{proof}

\begin{proposition}\label{Prop:circleinside}
Given three circles in $ \mathfrak{C} $, an intersection point lies at the interior of the circle centered at the same region that is bounded by the secant which joins the centers of the other two circles.
\end{proposition}

\begin{proof}
Without loss of generality, we assume that we have the arrangement of circles as illustrated in \Figs{Fig:InclusionProof}.
Let the secant $ (C_1C_2) $ which joins the center of the circles $ \mathcal{C}_1$ and $ \mathcal{C}_2 $ divides the reference circle into two oriented arcs, and let $ I,I' $ be the corresponding intersections. This secant also divides the plane into two regions such that each of them contains one intersection point of these two circles. Let the third circle $ \mathcal{C}_3 $ be centered at $ \arc{C_2C_1}$. We clearly see that $ C_3I>r $. Since an intersection point has to lie in the interior of $ \mathcal{C}_3 $, then the only candidate must be $ I' $ which is located at the same region as $ \arc{C_2C_1}$.	
\begin{figure}
\centering
\includegraphics[width=0.425\linewidth]{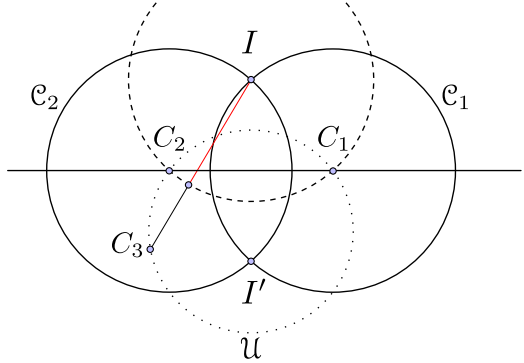}
\caption{Circle centered at $ C_3 $ only contains the nearest intersection point of the circle $ \mathcal{C}_1 $ and $ \mathcal{C}_2 $.}
\label{Fig:InclusionProof}
\end{figure}
\end{proof}
Notice that the result remains valid when the centers of the two circles are antipodal.
\begin{theorem}
No three circles in $ \mathfrak{C}$ share a common point.
\end{theorem}
\begin{proof}
If three circles in $ \mathfrak{C}$ share a common point, then their centers are respectively equidistant to this same point. Necessarily we have $ r=1 $, which is a contradiction.
\end{proof}

\begin{definition}
Let $ \mathfrak{C}_n:=\{\mathcal{C}_1, \mathcal{C}_2,\mathcal{C}_3,\ldots,\mathcal{C}_n\}$, $ n>1 $,
with the circles being arranged in counter-clockwise ascending order with respect to their index. A \textit{lune} \cite{APS} is the region -- the crescent-shaped slice -- delimited by $ \mathcal{D}_{i}\setminus \mathcal{D}_{i-1}$, i.e., the region inside one but outside the other. There are $ n $ lunes in this arrangement, namely $ \mathcal{D}_{i}\setminus \mathcal{D}_{i-1}$ for $i=2,\ldots n$ and $\mathcal{D}_{1}\setminus \mathcal{D}_{n}$. The condition we impose on the arrangement of circles is to ensure no circles are centered along the arc $ \arc{C_{i}C_{i-1}} $.
\end{definition}

\begin{proposition}
No arc of circles meet inside a lune.
\end{proposition}

\begin{proof}
Let $ k $ and $\ell $ be two nonnegative integers such that $ k<\ell $, and assume in the contrary that there exist a  lune $ \mathcal{D}_{i}\setminus \mathcal{D}_{i-1}$ that contains one of the intersection points of $ \mathcal{C}_k$ and $ \mathcal{C}_\ell $. Then the opposite intersection lies inside the circle $ \mathcal{C}_{i-1}$ by \hyperref[Prop:int3]{Proposition~\ref*{Prop:int3}}. Now, $ \mathcal{C}_k\in\mathcal{D}_i\setminus\mathcal{D}_{i-1} $ means that $ \mathcal{C}_k\in\mathcal{D}_i $, i.e., by \hyperref[Prop:circleinside]{Proposition~\ref*{Prop:circleinside}}, the circle $ \mathcal{C}_i $ is either centered at the arc $ \arc{C_kC_\ell} $ or $ \arc{C_\ell C_k} $. In any of these cases, one of the circle $ \mathcal{C}_k $ or $ \mathcal{C}_\ell $ is located along the arc $ \arc{C_iC_{i-1}} $, which is not conform to the definition of a lune.
\end{proof}

\begin{corollary}\label{cor:nregion}
Let $ n\geq2 $. Then each lune are divided into $ n-1 $ region. 
\end{corollary}

\begin{proof}
The border of a lune contains the two intersection points of its two circles. Therefore, there are $ n-2 $ circles that must contain one of these intersection points, and leaving at the same time $ n-2 $ non-crossing arcs through the lune. Consequently, the lune is divided into $ n-1 $ regions.
\end{proof}

\begin{corollary}
Let $ n\geq 1 $. The circles in $ \mathfrak{C}_n $ divide the plane into $ n^2-n+2 $ regions.
\end{corollary}

\begin{proof}
There are $ n $ lunes on the plane, each of which consists of $ n-1 $ regions. Then, the total regions is $ n(n-1) $ plus the outer region -- the plane itself -- and the central ``hole'', i.e., the intersection of the interiors of all the circles in $ \mathfrak{C}_n $. Therefore, there are in sum $ n(n-1)+2 $ regions.
\end{proof}

We conclude that the circles in $ \mathcal{C}_n $ are in general arrangement. If the centers describe a regular polygon, then we call the motif defined by $ \mathfrak{C}_n $ a \textit{circular rosette}. Rosin \cite{Rosin} defines the circular rosette as a geometrical figure formed by taking copies of a circle and rotating them about a point -- the rosette's center. For example, we draw a circular rosette made up from $ 6 $ circles in \Figs{fig:rosette}. The number of regions defined by these $ 6 $ circles is then $ P(6)=32 $.
\begin{figure}
\centering
\includegraphics[width=0.35\linewidth]{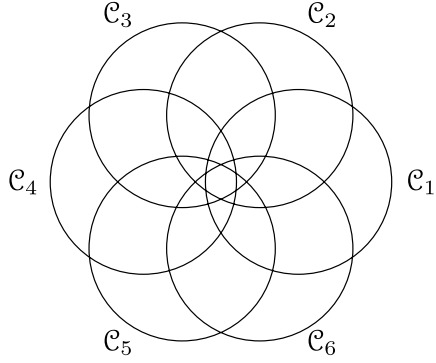}
\caption{A rosette of $ 6 $ circles}
\label{fig:rosette}
\end{figure}

\section{Encoding the partition of the plane}\label{Sec:RosetteEncoding}

We consider the binary alphabet $\{0,1\} $, and we let $ {\mathcal{P}}_n $ define the set of words of length $ n $ defined over the alphabet $ \{0,1\} $ such that a word $ \pi\in{\mathcal{P}}_n $ encodes a region of the plane described by the rosette of $ n $ circles. The usual operation of words is the \textit{concatenation}, that is, writing words as a compound. If we write a word as $ \pi=\sigma_1\sigma_2\cdots\sigma_n $, then the bit $ \sigma_{i}=1 $ (resp. $ \sigma_{i}=0 $) indicates that the concerned region lies inside (resp. outside) the $ i$-th circle. Moreover, we associate the plane without circles with the empty word $\varepsilon $. For $n=1,2,3 $ we encode the regions of the plane as in \Figs{Fig:123CirclesEncoding}. 

\begin{figure}
\centering
\includegraphics[width=\linewidth]{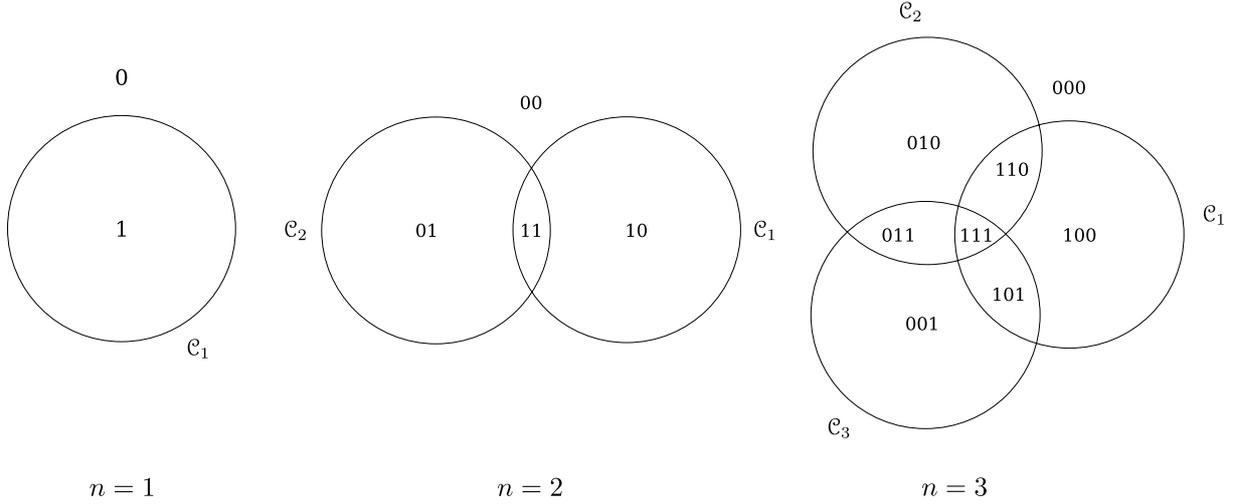}
\caption{The encodings of the partition of the plane defined by $ n $ circles, $ n=1,2,3 $.}
\label{Fig:123CirclesEncoding}
\end{figure}

\begin{proposition} Let $ \mathcal{L}_i $ denote the set of words associated with the regions of the lune $ \mathcal{D}_{i}\setminus \mathcal{D}_{i-1}$. Then a word $\pi_{i,j} $ in the set $ \mathcal{L}_i $ is of the form
\[\pi_{i,j}=\sigma_1\sigma_2\cdots\sigma_{i-1}\sigma_i\cdots\sigma_n,\ 0\leq j\leq n-2\]
where
\begin{displaymath}
\begin{cases}
\sigma_{i}= \sigma_{i+1}=\cdots=\sigma_{i+j}=1, & \textit{if $ i+j \leq n$},\\
\sigma_{i}= \sigma_{i+1}=\cdots=\sigma_{n}=\sigma_{1}=\sigma_{2}=\cdots=\sigma_{i+j-n}=1, & \textit{if $ n<i+j\leq n+i-2$},\\
\sigma_{\ell}=0, & \textit{elsewhere}.\\
\end{cases}
\end{displaymath}
\end{proposition}

\begin{proof}
Since there are $ n-2 $ arcs which pass through a lune, they are arranged with respect to the order of their center and  divide the lune into $ n- 1$ regions (see \Figs{Fig:LuneEncoding}). A region belongs to the interior of the circle $ \mathcal{C}_j $ for each $ j=i,i+1,\ldots,n-1,n,1,2,\ldots,i-2 $.
\begin{figure}
\centering
\includegraphics[width=0.775\linewidth]{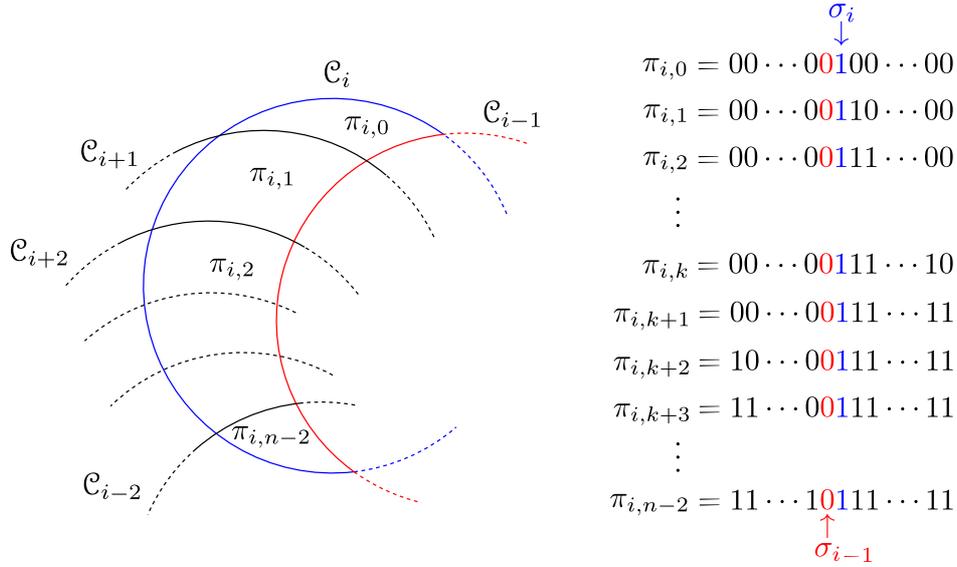}
\caption{The encodings of the regions of the lune $ \mathcal{D}_{i}\setminus \mathcal{D}_{i-1}$.}
\label{Fig:LuneEncoding}
\end{figure}
\end{proof} 

\begin{lemma}\label{lemma:newlune}
Given the family $ \mathfrak{C}_n $, by adding a new circle $ \mathcal{C}_{n+1} $ centered along the arc \textnormal{$ \arc{C_nC_1} $}, we add a new lune $ \mathcal{D}_{1}\setminus\mathcal{D}_{n+1} $, and create a new region inside each already existing lune. 
\end{lemma}

\begin{proof}
The new circle is centered at the arc $ \arc{C_nC_1} $, therefore, the new lune $ \mathcal{C}_1\setminus\mathcal{C}_{n+1} $ is created (see \Figs{Fig:RosetteNewRegion}). Since an arc of this circle does not intersect any other arc inside a lune, and since the interior of this circle has to contain an intersection point of each pair of the other circles, necessarily a new region is created inside each existing lune. 
\begin{figure}
\centering
\includegraphics[width=0.7\linewidth]{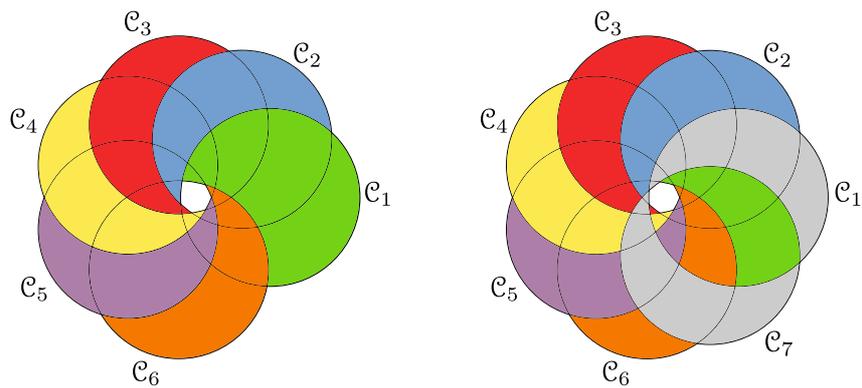}
\caption{The insertion of a new circle creates a new lune and introduce one new region inside each of the previous lunes.}
\label{Fig:RosetteNewRegion}
\end{figure}
\end{proof}

This lemma is the key ingredient in the following encoding formula.
\begin{theorem}\label{Thm:rulesplane}
Let $ \mathcal{P}_n $ be the set of words encoding the regions of the rosette of $ n $ circles. The set $ \mathcal{P}_{n+1} $ can be constructed with the help of the encodings of the previous region as well as those of the new regions that is obtained from the insertion of the new circle $ \mathcal{C}_{n+1}$. The set $ \mathcal{P}_{n+1} $ is constructed according as follows for $\pi=\sigma_1\sigma_2\cdots\sigma_n\in\mathcal{P}_n $. 
\begin{enumerate}
\item If $ \sigma_1=0 $ then $ \pi'=\pi0 \in\mathcal{P}_{n+1}$.
\item If $ \sigma_1=1 $ then $ \pi'=\pi1 \in\mathcal{P}_{n+1}$.
\item The additional regions are encoded by
\[
\begin{array}{rcl}
\left\{
\begin{array}{c}
0111\cdots111\\
0011\cdots111\\
0001\cdots111\\
\vdots\\

0000\cdots011\\
0000\cdots001\\
\end{array}
\right.
&and&
\left\{
\begin{array}{c}
1000\cdots000\\
1100\cdots000\\
1110\cdots000\\
\vdots\\
1111\cdots100\\
1111\cdots110\\
\end{array}
\right.
\end{array}
\]
\end{enumerate}
\end{theorem}

\begin{proof} 
First, note that the rules applies for the inner region $ 1\ldots111 $ and the outer region $0\ldots000 $. Now, adding a new circle introduced the new lune $ \mathcal{D}_1\setminus\mathcal{D}_{n+1} $ which is crossed by $ n-1 $ arcs. The corresponding regions are therefore encoded by
\[
\left\{
\begin{array}{c}
1000\cdots00\textcolor{red}{0} \\
1100\cdots00\textcolor{red}{0} \\
1110\cdots00\textcolor{red}{0} \\
\vdots\\
1111\cdots10\textcolor{red}{0} \\
1111\cdots11\textcolor{red}{0} \\
\end{array}
\right.
\]

The $ (n+1) $-th circle also crosses all the previous lunes. Since $ C_{n+1} \in\arc{C_nC_1}$ and by \hyperref[lemma:newlune]{Lemma~\ref*{lemma:newlune}}, some of the regions of these lunes lie inside the disk $ \mathcal{D}_{n+1} $, and some are outside. Moreover, by the same \hyperref[lemma:newlune]{Lemma}, one additional region appears inside each lune. So excluding the lune $ \mathcal{D}_1\setminus\mathcal{D}_{n+1} $, the encodings of regions of the current lunes are obtained by the following formula for each $ i=2,3,\ldots,n+1 $:

\[  \begin{tikzpicture}
    \matrix [matrix of math nodes] (m)
{
\sigma_1 & \sigma_2 & \sigma_3 &\cdots & \sigma_{i-2} & \sigma_{i-1} & \sigma_{i} & \sigma_{i+1} & \sigma_{i+2} & \cdots & \sigma_{n-1} & \sigma_{n} & \sigma_{n+1} \\
0 & 0 & 0 & \cdots & 0 & \textcolor{red}{0} & \textcolor{blue}{1} & 0 & 0 & \cdots & 0 & 0 & \textcolor{red}{0}\\
0 & 0 & 0 & \cdots & 0 & \textcolor{red}{0} & \textcolor{blue}{1} & 1 & 0 & \cdots & 0 & 0 & \textcolor{red}{0}\\
0 & 0 & 0 & \cdots & 0 & \textcolor{red}{0} & \textcolor{blue}{1} & 1 & 1 & \cdots & 0 & 0 & \textcolor{red}{0}\\
\vdots&\vdots&\vdots&\vdots&\vdots&\vdots&\vdots&\vdots&\vdots&\ddots&\vdots&\vdots&\vdots\\
0 & 0 & 0 & \cdots & 0 & \textcolor{red}{0} & \textcolor{blue}{1} & 1 & 1 & \cdots & 1 & 0 & \textcolor{red}{0}\\
0 & 0 & 0 & \cdots & 0 & \textcolor{red}{0} & \textcolor{blue}{1}& 1 & 1 & \cdots & 1 & 1 & \textcolor{red}{0}\\
\mathbf{0} & \mathbf{0} & \mathbf{0} & \cdots & \mathbf{0} & \mathbf{0} & \mathbf{1} & \mathbf{1} & \mathbf{1} & \cdots & \mathbf{1} & \mathbf{1} & \mathbf{1}\\
1 & 0 & 0 & \cdots & 0 & \textcolor{red}{0} & \textcolor{blue}{1} & 1 & 1 & \cdots & 1 & 1 & \textcolor{blue}{1}\\
 1 & 1 & 0 & \cdots & 0 & \textcolor{red}{0} & \textcolor{blue}{1} & 1 & 1 & \cdots & 1 & 1 & \textcolor{blue}{1}\\
 1 & 1 &  1 & \cdots & 0 & \textcolor{red}{0} & \textcolor{blue}{1} & 1 & 1 & \cdots & 1 & 1 & \textcolor{blue}{1}\\
\vdots&\vdots&\vdots&\ddots&\vdots&\vdots&\vdots&\vdots&\vdots&\vdots&\vdots&\vdots&\vdots\\
 1 & 1 &  1 & \cdots & 1 & \textcolor{red}{0} & \textcolor{blue}{1} & 1 & 1 & \cdots & 1 & 1 & \textcolor{blue}{1}\\
    }; 
    \draw[color=red] (m-2-1.north west) -- (m-2-12.north east) -- (m-7-12.south east) -- (m-7-1.south west) -- (m-2-1.north west);
    \draw[color=red] (m-9-1.north west) -- (m-9-12.north east) -- (m-13-12.south east) -- (m-13-1.south west) -- (m-9-1.north west);
  \end{tikzpicture}\]
Thus, letting $ i=2\ldots n+1 $, we identify the new regions inside each lune $ \mathcal{D}_{i}\setminus\mathcal{D}_{i-1} $ , namely
\[
\left\{
\begin{array}{c}
0\mathbf{111}\cdots\mathbf{1111}\\
00\mathbf{11}\cdots\mathbf{1111}\\
000\mathbf{1}\cdots\mathbf{1111}\\
\vdots\\
0000\cdots{00}\mathbf{11}\\
0000\cdots00{0}\mathbf{1}.
\end{array}
\right.
\]
We point out that in the previous calculation we do not consider the lune $ \mathcal{D}_{n+1}\setminus\mathcal{D}_{n} $ as ``new'', but rather as an ``extension'' of the previously lune $ \mathcal{D}_{1}\setminus\mathcal{D}_{n} $ as formulated by \hyperref[lemma:newlune]{Lemma~\ref*{lemma:newlune}}.
\end{proof}

In order to handle a more explicit encoding formula, we should introduce the following notation.
\begin{notation}
Define the $ n^{\mbox{\scriptsize th}}$ concatenation of a word $ w $ as 
\begin{equation*}
w^n=\underbrace{ww\cdots w}_{\text{$ n $ copies}},
\end{equation*}
with $ w^1=w $ and $ w^0=\varepsilon $. 
\end{notation}

We shall make extensive use of these operations  throughout the rest of this paper. The next corollary now directly follows from \hyperref[Thm:rulesplane]{Theorem~\ref*{Thm:rulesplane}}.
\begin{corollary}\label{Cor:definitionPn}
The partition of the plane divided by $ n $ circles is encoded as
\begin{equation}\label{eq:bitonicdef}
\mathcal{P}_1=\{0,1\},\ \mathcal{P}_2=\{00,01,10,11\},\ \mathcal{P}_n:=\left\{0^n,1^n\right\}\cup\mathcal{P}_n^{\mathtt{01}}\cup\mathcal{P}_n^{\mathtt{10}}\cup\mathcal{P}_n^{\mathtt{00}}\cup\mathcal{P}_n^{\mathtt{11}},\ n\geq 3 
\end{equation}
where
\begin{align*}
\mathcal{P}_n^{\mathtt{00}}&=\left\{0^k1^{n-k}\ | \ 1\leq k\leq n-1\right\};\\
\mathcal{P}_n^{\mathtt{11}}&=\left\{1^k0^{n-k}\ | \ 1\leq k\leq n-1\right\};\\
\mathcal{P}_n^{\mathtt{01}}&=\left\{0^k1^{n-p-k}0^p\ |\ 1\leq p\leq n-2\mbox{ and } 1\leq k\leq n-p-1\right\};\\
\mathcal{P}_n^{\mathtt{10}}&=\left\{1^k0^{n-p-k}1^p\ |\ 1\leq p\leq n-2\mbox{ and } 1\leq k\leq n-p-1\right\}.
\end{align*}
\end{corollary}
\begin{proof}
For $ n=1,2 $, we have the corresponding encodings as illustrated in \Figs{Fig:123CirclesEncoding}, namely $ \mathcal{P}_1=\{0,1\} $ and $ \mathcal{P}_2=\{00,01,10,11\} $.
For $ n=3$, we obtain $ \mathcal{P}_3^{\mathtt{00}}=\{011,001\} $, $ \mathcal{P}_3^{\mathtt{11}}=\{110,100\} $, $ \mathcal{P}_3^{\mathtt{01}}=\{010\} $, $ \mathcal{P}_3^{\mathtt{01}}=\{101\} $ and the inner plus the outer region $ \{000,111\} $. These results also match the encodings in \Figs{Fig:123CirclesEncoding}. 

Assume that the formula \eqref{eq:bitonicdef} holds when $ n=\ell>3 $ and let us show that it still holds when $ n=\ell+1 $. By \hyperref[Thm:rulesplane]{Theorem~\ref*{Thm:rulesplane}} and the induction hypothesis, we have 
\begin{equation*}
\mathcal{P}_{\ell+1}=\Big(\mathcal{P}_\ell^{\mathtt{01}}\cup\mathcal{P}_\ell^{\mathtt{00}}\cup\left\{0^\ell\right\}\Big)0\cup\Big(\mathcal{P}_\ell^{\mathtt{10}}\cup\mathcal{P}_\ell^{\mathtt{11}}\cup\left\{1^\ell\right\}\Big)1\cup\mathcal{P}_{\ell+1}^{\mathtt{00}}\cup\mathcal{P}_{\ell+1}^{\mathtt{11}},
\end{equation*}
where
\begin{align*}
\mathcal{P}_\ell^{\mathtt{01}}0&=\left\{0^k1^{\ell-p-k}0^{p+1}\ |\ 1\leq p\leq \ell-2\mbox{ and } 1\leq k\leq \ell-p-1\right\}\\
&=\left\{0^k1^{\ell-{\widehat{p}}+1-k}0^{\widehat{p}}\ |\ 2\leq \widehat{p}\leq \ell-1\mbox{ and } 1\leq k\leq \ell-\widehat{p}\right\}
\end{align*}
and 
\begin{align*}
\mathcal{P}_\ell^{\mathtt{00}}0&=\left\{0^k1^{\ell-k}0\ | \ 1\leq k\leq \ell-1\right\}.
\end{align*}
We deduce that
\begin{align*}
\mathcal{P}_\ell^{\mathtt{00}}0\cup\mathcal{P}_\ell^{\mathtt{01}}0&=\left\{0^k1^{\ell-{\widehat{p}+1}-k}0^{\widehat{p}}\ |\ 1\leq \widehat{p}\leq \ell-1\mbox{ and } 1\leq k\leq \ell-\widehat{p}\right\}\\
&=\mathcal{P}_{\ell+1}^{\mathtt{01}}.
\end{align*}
Correspondingly, we have 
\begin{align*}
\mathcal{P}_\ell^{\mathtt{11}}1\cup\mathcal{P}_\ell^{\mathtt{10}}1&=\left\{1^k0^{\ell-{\widehat{p}+1}-k}1^{\widehat{p}}\ |\ 1\leq \widehat{p}\leq \ell-1\mbox{ and } 1\leq k\leq \ell-\widehat{p}\right\}\\
&=\mathcal{P}_{\ell+1}^{\mathtt{10}}.
\end{align*}
Since $ \left\{0^\ell\right\}0=\left\{0^{\ell+1}\right\} $ and $ \left\{1^\ell\right\}1=\left\{1^{\ell+1}\right\} $, we finally have
\begin{equation*}
\mathcal{P}_{\ell+1}=\mathcal{P}_{\ell+1}^{\mathtt{01}}\cup\mathcal{P}_{\ell+1}^{\mathtt{10}}\cup\mathcal{P}_{\ell+1}^{\mathtt{00}}\cup\mathcal{P}_{\ell+1}^{\mathtt{11}}\cup\left\{0^{\ell+1},1^{\ell+1}\right\}.
\end{equation*}
\end{proof}

\begin{corollary} For $ n\geq 1 $, we recursively define $ \mathcal{P}_{n+1} $ as follows.
\begin{equation}\label{eq:recplane}
\mathcal{P}_{n+1}=\left\{0^{n+1},1^{n+1}\right\}\cup\left(\bigcup\limits_{p=0}^{n-1}\mathcal{P}_{n-p+1}^{\mathtt{00}}0^{p}\right)\cup\left(\bigcup\limits_{p=0}^{n-1}\mathcal{P}_{n-p+1}^{\mathtt{11}}1^{p}\right).
\end{equation}
\end{corollary}

\begin{proof}
When $ n=1 $, we have
\begin{align*}
\mathcal{P}_{2}&=\left\{0^{2},1^{2}\right\}\cup\left(\bigcup\limits_{p=0}^{0}\mathcal{P}_{1-p+1}^{\mathtt{00}}0^{p}\right)\cup\left(\bigcup\limits_{p=0}^{1-1}\mathcal{P}_{n-p+1}^{\mathtt{11}}1^{p}\right)\\
&=\{00,11,01,10\}.
\end{align*}
Now let $ n\geq 2 $. By \hyperref[Cor:definitionPn]{Corollary~\ref*{Cor:definitionPn}}, we may combine the pairs $ \mathcal{P}_{n+1}^{\mathtt{00}}$, $\mathcal{P}_{n+1}^{\mathtt{01}} $ and $ \mathcal{P}_{n+1}^{\mathtt{10}}$, $\mathcal{P}_{n+1}^{\mathtt{11}} $ as
\begin{align*}
\mathcal{P}_{n+1}^{\mathtt{00}}\cup\mathcal{P}_{n+1}^{\mathtt{01}}&= \left\{0^k1^{n-p-k+1}0^p\ |\ 0\leq p\leq n-1\mbox{ and } 1\leq k\leq n-p\right\}\\
&=\bigcup\limits_{p=0}^{n-1}\mathcal{P}_{n-p+1}^{\mathtt{00}}0^{p}
\end{align*}
and
\begin{align*}
\mathcal{P}_{n+1}^{\mathtt{10}}\cup\mathcal{P}_{n+1}^{\mathtt{11}}&= \left\{1^k0^{n-p-k+1}1^p\ |\ 0\leq p\leq n-1\mbox{ and } 1\leq k\leq n-p\right\}\\
&=\bigcup\limits_{p=0}^{n-1}\mathcal{P}_{n-p+1}^{\mathtt{11}}1^{p}.
\end{align*}
We conclude by writing
\begin{equation*}
\mathcal{P}_{n+1}=\left\{0^{n+1},1^{n+1}\right\}\cup\left(\mathcal{P}_{n+1}^{\mathtt{00}}\cup\mathcal{P}_{n+1}^{\mathtt{01}}\right)\cup\left(\mathcal{P}_{n+1}^{\mathtt{10}}\cup\mathcal{P}_{n+1}^{\mathtt{11}}\right).
\end{equation*}
\end{proof}

\begin{remark}[Lang \cite{Lang}]
A sequence $ a = a_0,a_{1},\ldots, a_{n-1}$ with $ a_{i}\in \{0, 1\}$, $i=0,\ldots,n-1$ is called a \textit{$ 0 $-$ 1 $-sequence}. A $ 0 $-$ 1 $-sequence is called \textit{bitonic} \cite{Batcher}, if it contains at most two changes between $ 0 $ and $ 1 $, i.e., if there exist subsequence lengths $ k,m\in\{1,\ldots,n\}$ such that
\[\begin{array}{lll}
a_0,\ldots, a_{k-1} = 0,& a_k, \ldots, a_{m-1} = 1 ,&  a_m, \ldots, a_{n-1} = 0\  \textit{or}\\
a_0, \ldots, a_{k-1} = 1,& a_k, \ldots, a_{m-1} = 0 ,&  a_m,\ldots, a_{n-1} = 1.
\end{array}\]

It follows that the set $ \mathcal{P}_{n+1} $ as defined in formulas \eqref{eq:bitonicdef} and \eqref{eq:recplane} is exactly the set of binary bitonic sequences of length $n+1 $. The cardinal of such set is known in the OEIS as $ \seqnum{A014206}(n)=n^2+n+2 $. Indeed, since $ \left\{0^\ell,1^n\right\}\cap\mathcal{P}_\ell^{\mathtt{01}}\cap\mathcal{P}_\ell^{\mathtt{10}}\cap\mathcal{P}_\ell^{\mathtt{00}}\cap\mathcal{P}_\ell^{\mathtt{11}}=\varnothing $ for any nonnegative number $ \ell $, then from formula \eqref{eq:recplane} we verify that
\begin{align*}
\#\mathcal{P}_{n+1}&=\#\left\{0^{n+1},1^{n+1}\right\}+\sum_{p=0}^{n-1}\#\mathcal{P}_{n-p+1}^{\mathtt{00}}0^{p}+\sum_{p=0}^{n-1}\#\mathcal{P}_{n-p+1}^{\mathtt{11}}1^{p}\\
&=2+2\sum_{p=0}^{n-1}(n-p)\\
&=n^2+n+2,\ n\geq 1.
\end{align*}

This formula is actually still valid for $ n=0 $ since $ \#\mathcal{P}_{1}=2 $. The value $ \#\mathcal{P}_{n+1} $ describes the maximal number of regions into which the plane is divided by $ n+1 $ circles in general arrangement. But recall on the other hand that it is also the number of $ 2 $-states that is obtained by splitting the crossings of the $ n $-twist knot. The following section is motivated with this enumeration.
\end{remark}
\section{Encoding the states of the twist knot}\label{Sec:StatesEncoding}

In the present section, we associate as well a state of a $ n $-crossing knot with a length $ n $ binary word. If we initially fix the order of splits, then the word $ \omega_n=\sigma_{1}\sigma_{2}\cdots\sigma_{n} $ relates the split we have applied at each crossing. The bit $ \sigma_{i}=0 $ corresponds to an $ A $-split at the $ i $-th crossing, while $ \sigma_{i}=1 $ corresponds to a $ B $-split. As an illustration, we consider the states of the trefoil $ \uptau_1 $ ($ 3 $ crossings) in \Figs{Fig:TrefoilStates}. 

We impose from here that the first two splits are always those of the link part for the twist knot. The order of these splits does not matter since the results are $ 00 $, $ 01 $, $ 10 $ and $ 11 $. Besides, we always proceed by splitting the leftmost crossing for the twist loop, and for the foil knot, the split would start with an arbitrary crossing and followed with the first-nearest clockwise one.
\begin{figure}
\centering
\includegraphics[width=.8\linewidth]{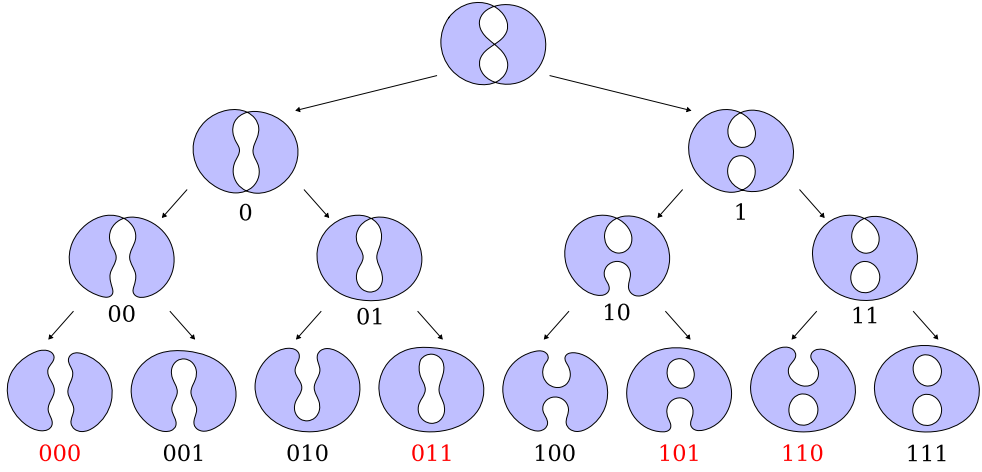}
\caption{The states of the trefoil knot and the associated encodings.}
\label{Fig:TrefoilStates}
\end{figure}

Let $ \mathtt{T}_{n} $ denote the set of encodings of the $ n $-twist knot. Thus, $ \mathtt{T}_{n} $ actually consists of a set of words of length $ n+2 $. It would cause no confusion if we take into account that the $ n $-twist knot is, recall it, associated with the number of its half-twists (here $ n $). For instance, the corresponding set associated with the trefoil $ \uptau_1 $ in \Figs{Fig:TrefoilStates} is given by  \[\mathtt{T}_{1}=\{000,001,010,011,100,101,110,111\}.\]
 
Let on the other hand $ \mathsf{F}_{n} $ and $ \mathsf{T}_{n} $ respectively denote the set of encodings of the $ n $-foil and the $ n $-twist loop. The split of the link part suggests the following formula
\begin{equation*}
\mathtt{T}_{n}:= 00\mathsf{T}_{n}\cup 01\mathsf{F}_{n}\cup 10\mathsf{F}_{n}\cup 11\mathsf{F}_{n},\ n\geq0
\end{equation*}
with obviously
\begin{equation*}
01\mathsf{F}_{n}\cap 10\mathsf{F}_{n}\cap 00\mathsf{T}_{n}\cap 11\mathsf{T}_{n}=\varnothing.
\end{equation*}
Referring back to \Figs{Fig:StateOfLinkPart}, it is immediate to notice that for $ n= 0$, we have
\begin{equation*}
\mathtt{T}_{0}:= \{00,01,10,11\},
\end{equation*}
so that $ \mathsf{T}_{0}=\mathsf{F}_{0}=\{\varepsilon\} $. 

\begin{remark}\label{rem:emptyword}
 If we let $ \mathtt{T}_{n,k} $, $ \mathsf{F}_{n,k} $ and $ \mathsf{T}_{n,k} $ respectively denote the subset of $ \mathtt{T}_{n} $, $ \mathsf{F}_{n} $ and $ \mathsf{T}_{n}$ which represent the encodings of a $ k $-state, then we have
\begin{equation}\label{eq:statesnk}
\mathtt{T}_{n,k}:= 00\mathsf{T}_{n,k}\cup 01\mathsf{F}_{n,k}\cup 10\mathsf{F}_{n,k}\cup 11\mathsf{F}_{n,k-1},\ n\geq0.
\end{equation}
The last factor in the right side is expressed with the $ 1 $-state because it coincides with the states of the disjoint union of a resulting unknot -- now counted as one component -- and the $ n $-foil. 
Moreover
\begin{equation}\label{eq:capnothin}
\bigcap\limits_{i>0}^{}\mathsf{T}_{n,i}=\bigcap\limits_{j>0}^{}\mathsf{F}_{n,j}=\bigcap\limits_{k>0}^{}\mathtt{T}_{n,k}=\varnothing,
\end{equation}
since the states are obtained form a full binary tree. The index $ i $, $ j $ and $ k $ in \eqref{eq:capnothin} are respectively supposed to follow the range defined in \eqref{eq:tnk}, \eqref{eq:fnk} and \eqref{eq:taunk}. Consequently, depending on the context, we refer to the empty word when we denote a state of a specified number of components. For instance, we write $ \mathsf{F}_{0,1}=\varnothing $ whereas $ \mathsf{F}_{0,2}=\{\varepsilon\} $ because the knot $ F_0 $ consists of two components.
\end{remark}

Since we are particularly interested in the class of the $ 2 $-state, then by \Figs{Fig:0TwsitKnot} and formula \eqref{eq:statesnk} we write
\begin{equation*}
\mathtt{T}_{0,2}:=\{01,10\},\ \mathtt{T}_{n,2}:= 00\mathsf{T}_{n,2}\cup 01\mathsf{F}_{n,2}\cup 10\mathsf{F}_{n,2}\cup 11\mathsf{F}_{n,1},\ n\geq1
\end{equation*}
Here we have necessarily $\mathsf{T}_{0,2}=\mathsf{F}_{0,1}=\varnothing $ as mentioned in \hyperref[rem:emptyword]{Remark~\ref*{rem:emptyword}}. 

\begin{lemma}\label{Lem:F1T2}
 For $ n\geq1 $, we have
\begin{equation*} 
\mathsf{F}_{n,1}=\mathsf{T}_{n,2}.
\end{equation*}
\end{lemma}
\begin{proof}
First of all, notice that they have the same cardinality. Referring back to formulas \eqref{eq:tnk} and \eqref{eq:fnk}, we have $\# \mathsf{F}_{n,1}=\#\mathsf{T}_{n,2} =\binom{n}{1}$. For both knots, we have to choose one crossing at which we perform an $ A $-split. We then have perform a $ B $-split at the remaining crossings. So the encodings must be the same.
\end{proof}

Henceforth, let us adopt the following notation: $ \mathtt{T}_{n,2}=\mathfrak{T}_{n} $, $ \mathsf{T}_{n,2}=\mathcal{T}_n $ and $ \mathsf{F}_{n,2}=\mathcal{F}_n $. Taking into consideration \hyperref[Lem:F1T2]{Lemma~\ref*{Lem:F1T2}}, we now write
\begin{equation}
\mathfrak{T}_0:=\{01,10\},\ \mathfrak{T}_n:= 01\mathcal{F}_n\cup 10\mathcal{F}_n\cup 00\mathcal{T}_n\cup 11\mathcal{T}_n,\ n\geq1.
\end{equation}

\begin{proposition}
The set of $ 2 $-states of the $ n $-twist loop is
\begin{equation}\label{eq:Tn}
\mathcal{T}_0:=\varnothing,\ \mathcal{T}_1:=\{0\},\ \mathcal{T}_n:=\left\{1^k01^{n-k-1}\mid 0\leq k\leq n-1\right\},\ n\geq2.
\end{equation}
\end{proposition}

\begin{proof}
When $ n=1 $, there is only one possibility to obtain a $ 2 $-state for the knot $ T_1 $, namely an $ A $-split. Thus $\mathcal{T}_1:=\{0\} $. When $ n=2 $, we have $\mathcal{T}_2:=\{01,10\} $ (see \Figs{Fig:12TwsitLoop}).

\begin{figure}
\centering
\includegraphics[width=.9\linewidth]{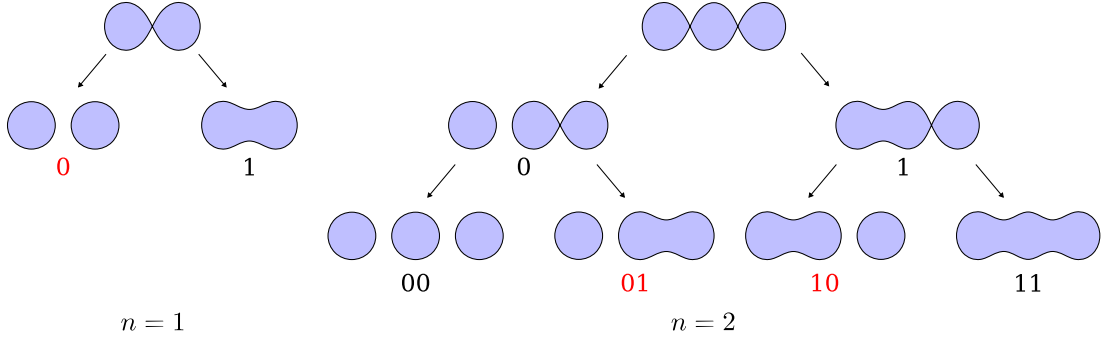}
\caption{The states of the $ n $-twist loop, $ n=1,2 $.}
\label{Fig:12TwsitLoop}
\end{figure}

Assume that formula \eqref{eq:Tn} holds when $ n=\ell >2$ and let us show that it still holds when $ n=\ell+1 $. When we split the leftmost crossing of the $ (\ell+1) $-twist loop, we obtain either a disjoint union of the unknot and a $ \ell $-twist loop, or simply a $ \ell$-twist loop (see \Figs{Fig:StateOfCrossingTwistLoop}). The first case corresponds to an $ A $-split, and in order to end up to two components, we must apply a $ B $-split a each of the crossings of the $ \ell$-twist loop part. Therefore we have
\begin{equation*}
\mathcal{T}_{\ell+1}=0\left\{1^\ell\right\}\cup1\mathcal{T}_{\ell}.
\end{equation*}
From the induction hypothesis we write
\begin{align*}
\mathcal{T}_{\ell+1}&=\left\{01^\ell\right\}\cup\left\{11^k01^{\ell-k-1}\mid 0\leq k\leq \ell-1\right\}\\
&=\left\{01^\ell\right\}\cup\left\{1^{\widehat{k}}01^{\ell-\widehat{k}}\mid 1\leq \widehat{k}\leq \ell\right\}\\
&=\left\{1^{\widehat{k}}01^{\ell-\widehat{k}}\mid 0\leq \widehat{k}\leq \ell\right\}.
\end{align*}.
\end{proof}

\begin{proposition}
The encodings of $ 2 $-states associated with the $ n $-foil is recursively given by 
\begin{equation}\label{eq:recfoil}
\mathcal{F}_0=\left\{\varepsilon\right\},\ \mathcal{F}_1=\left\{1\right\},\ \mathcal{F}_{n}=0\mathcal{T}_{n-1}\cup1\mathcal{F}_{n-1},\ n\geq2.
\end{equation} 
Moreover, for $ n\geq 2 $, we define the set $ \mathcal{F}_{n} $ by 
\begin{equation}\label{eq:foilset}
\mathcal{F}_{n}:=\left\{1^p01^k01^{n-p-k-2}\mid 0\leq p\leq n-2\ and \ 0\leq k\leq n-p-2\right\}\cup \left\{1^{n}\right\}.
\end{equation}
\end{proposition}

\begin{proof}
When $ n=1 $, an $ A $-split produces one $ 1 $-state whereas a $ B $-split produces one $ 2 $-state. Accordingly, we have $\mathcal{F}_1=\left\{1\right\} $ (see \Figs{fig:hopfstates}).

\begin{figure}
\centering
\includegraphics[width=.7\linewidth]{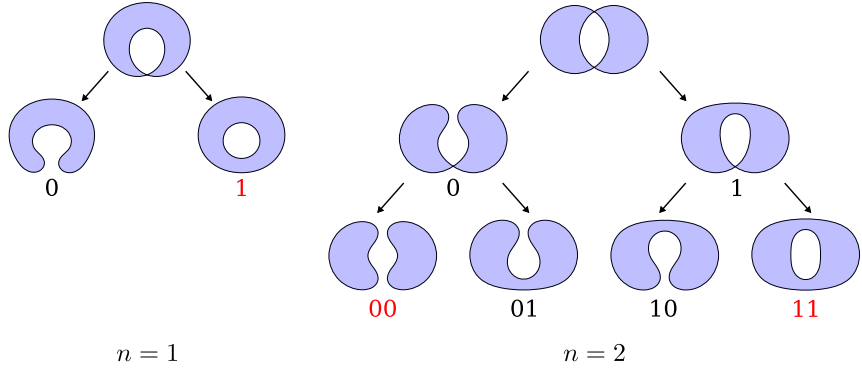}
\caption{The states of the $ n $-foil, $ n=1,2 $.}
\label{fig:hopfstates}
\end{figure}

When $ n\geq 2 $, recall that the generating polynomial of the $ n $-foil is given by \[F_n(x)=T_{n-1}(x)+F_{n-1}(x).\]
The first term in the right side is associated with an $ A $-split and the second term with a $ B $-split. Expressing such polynomial in terms of the class of the $ 2 $-state, we have
\begin{equation*}
\mathcal{F}_{n}=0\mathcal{T}_{n-1}\cup1\mathcal{F}_{n-1}.
\end{equation*}

Let us now use induction to prove that formula \eqref{eq:foilset} holds. When $ n=2 $, the boundaries $ 0\leq p\leq 2-2 $ and $ 0\leq k\leq 2-p-2 $ imply \[\mathcal{F}_2=\{00,11\}.\] 

By \Figs{fig:hopfstates}, we have $ \mathsf{F}_3=\{00,01,10,11\} $ where the subset $ \{00,11\} $ corresponds to two consecutive $ A $-splits and two consecutive $ B $-splits.

Let $ \ell\in\mathbb{N} $ be given, and assume that the formula \eqref{eq:foilset} holds for $ n = \ell $. From the recurrence \eqref{eq:recfoil}, we write
\begin{equation*}
\mathcal{F}_{\ell+1}=1\mathcal{F}_{\ell}\cup0\mathcal{T}_\ell.
\end{equation*}
By the induction hypothesis, we have
\begin{equation}\label{eq:OT}
0\mathcal{T}_\ell:=\left\{01^k01^{\ell-k-1}\mid0\leq k\leq \ell-1\right\}
\end{equation}
and 
\begin{align}
1\mathcal{F}_{\ell}&:=\left\{1^{p+1}01^k01^{\ell-p-k-2}\mid0\leq p\leq \ell-2\ and\ \leq k\leq\ell-p-2\right\}\cup \left\{11^{\ell}\right\}\nonumber\\
&=\left\{1^{\widehat{p}}01^k01^{\ell-\widehat{p}-k-1}\mid 1\leq \widehat{p}\leq \ell-1\ and\ 0\leq k\leq\ell-\widehat{p}-1\right\}\cup \left\{1^{\ell+1}\right\}.\label{eq:1F}
\end{align}
Combining \eqref{eq:OT} and \eqref{eq:1F} we get
\begin{equation*}
\mathcal{F}_{\ell+1}=\left\{1^{\widehat{p}}01^k01^{\ell-\widehat{p}-k-1}\mid0\leq \widehat{p}\leq \ell-1\ and\ 0\leq k\leq\ell-\widehat{p}-1\right\}\cup \left\{1^{\ell+1}\right\}.
\end{equation*}
\end{proof}

\begin{lemma}\label{Lem:lemmtau}
When $ n\geq 1 $, we have
\begin{equation*}
\mathfrak{T}_{n+1}=\big(011\mathcal{F}_n\cup 101 \mathcal{F}_n \cup 010\mathcal{T}_n\cup 100\mathcal{T}_n\big)\cup 00\mathcal{T}_{n+1}\cup 11\mathcal{T}_{n+1}.
\end{equation*}
\end{lemma}

\begin{proof}
Taking into consideration formula \eqref{eq:recfoil}, we have
\begin{align*}
\mathfrak{T}_{n+1}&=01\mathcal{F}_{n+1}\cup 10\mathcal{F}_{n+1}\cup 00\mathcal{T}_{n+1}\cup 11\mathcal{T}_{n+1}\\
&=
01\big(0\mathcal{T}_n\cup 1\mathcal{F}_n\big)\cup
10\big(0\mathcal{T}_n\cup 1\mathcal{F}_n\big)\cup 
00\mathcal{T}_{n+1}\cup
11\mathcal{T}_{n+1}\\
&=\big(011\mathcal{F}_n\cup 101 \mathcal{F}_n \cup 010\mathcal{T}_n\cup 100\mathcal{T}_n\big)\cup 00\mathcal{T}_{n+1}\cup 11\mathcal{T}_{n+1}.
\end{align*}
\end{proof}

\begin{corollary}
Let $ \omega \in \mathfrak{T}_n $ such that $ \omega $ is a compound of a prefix $ \alpha\in\{00,01,10,11\} $ and a factor $ \beta $ from either $\mathcal{F}_{n}$ or $\mathcal{T}_{n}$. 
Let $ \mathfrak{p} $ be a transformation defined as follows.
\[
\mathfrak{p}(01)=011,\ \mathfrak{p}(10)=101,\ \mathfrak{p}(00)=010\ \mbox{and}\ \mathfrak{p}(11)=100.
\]
We also define a map $ \uppsi  $ which operates on the suffix of a word in $\mathfrak{T}_n $,
\begin{align*}
\uppsi :\mathfrak{T}_n&\longrightarrow \{0,1\}^{n+1}:=\big\{\sigma_1\sigma_2\cdots\sigma_n\sigma_{n+1}\mid\sigma_{i}\in\{0,1\}\big\}\\
\omega=\alpha\beta&\longmapsto \uppsi (\omega)=\mathfrak{p}(\alpha)\beta,\ \textit{with}\ \alpha\in\{00,01,10,11\}.
\end{align*}
Then we have
\begin{equation}\label{eq:taupsi}
\mathfrak{T}_{n+1}=\uppsi (\mathfrak{T}_n)\cup\big( 00\mathcal{T}_{n+1}\cup 11\mathcal{T}_{n+1}\big).
\end{equation}
\end{corollary}

\begin{proof}
Notice that the map $\uppsi $ is injective. Indeed, for $ \omega $ and $ \omega' $ in $ \mathfrak{T}_n $ the equality $ \uppsi (\omega)=\uppsi (\omega') $ is word equality, i.e., both side are considered equal only if they are made up of the same number of letter and the same letters at identical positions. It follows that $ \omega=\omega'$. Now consider the restriction of the set $\mathfrak{T}_{n+1}$ to $ 011\mathcal{F}_n\cup 101 \mathcal{F}_n \cup 010\mathcal{T}_n\cup 100\mathcal{T}_n $. It is clear that the map $ \uplambda $ defined as follows is a bijection.
\[
\begin{array}{rcl}
\uplambda:01\mathcal{F}_{n}\times 10\mathcal{F}_{n}\times 00\mathcal{T}_{n}\times 11\mathcal{T}_{n}&\longrightarrow&
011\mathcal{F}_n\times 101 \mathcal{F}_n \times 010\mathcal{T}_n\times 100\mathcal{T}_n\\
\left(01\beta_1,10\beta_2,00\beta_3,11\beta_4\right)&\longmapsto&\left(011\beta_1,101\beta_2,010\beta_3,100\beta_4\right).
\end{array}
\]
Since $ 01\mathcal{F}_{n}\cap 10\mathcal{F}_{n}\cap 00\mathcal{T}_{n}\cap 11\mathcal{T}_{n}=\varnothing$, it follows immediately that 
\[
\uppsi (\mathfrak{T}_n)=011\mathcal{F}_n\cup 101 \mathcal{F}_n \cup 010\mathcal{T}_n\cup 100\mathcal{T}_n,
\]
and we conclude by \hyperref[Lem:lemmtau]{Lemma~\ref*{Lem:lemmtau}}.
\end{proof}

\begin{example}
Consider the set $ \mathfrak{T}_1 =\{\mathbf{00}0,\mathbf{01}1,\mathbf{10}1,\mathbf{11}0\} $. We have 
\begin{equation*}
\uppsi (\mathfrak{T}_1)=\{\mathbf{010}0,\mathbf{011}1,\mathbf{101}1,\mathbf{100}0\}
\end{equation*}
and
\begin{equation*}
\mathcal{T}_2=\{01,10\} .
\end{equation*}
Then applying formula \eqref{eq:taupsi}, we have
\[
\mathfrak{T}_2=\{0100,0111,1011,1000,0001,0010,1101,1110\}.
\]
\end{example}

\begin{proposition}\label{Prop:psimap}
When $ n\geq 1 $, the set $ \mathfrak{T}_n $ is obtained by the formula
\begin{equation}\label{eq:psimap}
\mathfrak{T}_{n}=\uppsi^{n}\big(\{01,10\}\big)\cup\left(\bigcup\limits_{p=0}^{n-1} \uppsi^{p}(00)\mathcal{T}_{n-p}\right)\cup\left(\bigcup\limits_{p=0}^{n-1} \uppsi^{p}(11)\mathcal{T}_{n-p}\right),
\end{equation}
where $ \uppsi^n= \uppsi \circ\uppsi \circ\cdots\circ\uppsi $ and $ \uppsi^0=Id $.
\end{proposition}

\begin{proof}
Let us use induction to prove that formula \eqref{eq:psimap} holds for any $ n\geq 1 $. 

When $ n=1 $, we have
\begin{align*}
\mathfrak{T}_{1}&=\uppsi^{1}\big(\{01,10\}\big)\cup\left(\bigcup\limits_{p=0}^0 \uppsi^{p}(00)\mathcal{T}_{1-p}\right)\cup\left(\bigcup\limits_{p=0}^0 \uppsi^{p}(11)\mathcal{T}_{1-p}\right)\\
&=\{011,101,000,110\}.
\end{align*}
Assume that the identity holds until $ n=\ell> 1$ and let us show that it still holds when $ n=\ell+1 $. Formula \eqref{eq:taupsi} allows us to write $ \mathfrak{T}_{\ell+1} $ as follows.
\begin{align*}
\mathfrak{T}_{\ell+1}&=\uppsi (\mathfrak{T}_\ell)\cup 00\mathcal{T}_{\ell+1}\cup11\mathcal{T}_{\ell+1}\\
&=\uppsi (\mathfrak{T}_\ell)\cup\uppsi ^{0}(00)\mathcal{T}_{\ell+1}\cup\uppsi ^{0}(11)\mathcal{T}_{\ell+1},
\end{align*}
where $ \mathfrak{T}_\ell $ is given by the induction hypothesis, i.e.,
\begin{align*}
\uppsi (\mathfrak{T}_\ell)&=\uppsi \left(\uppsi^{\ell}\big(\{01,10\}\big)\cup\left(\bigcup\limits_{p=0}^{\ell-1} \uppsi^{p}(00)\mathcal{T}_{\ell-p}\right)\cup\left(\bigcup\limits_{p=0}^{\ell-1} \uppsi^{p}(11)\mathcal{T}_{\ell-p}\right)\right)\\
&=\uppsi^{\ell+1}\big(\{01,10\}\big)\cup\left(\bigcup\limits_{p=0}^{\ell-1} \uppsi^{p+1}(00)\mathcal{T}_{\ell-p}\right)\cup\left(\bigcup\limits_{p=0}^{\ell-1} \uppsi^{p+1}(11)\mathcal{T}_{\ell-p}\right).
\end{align*}
Thus
\begin{align*}
\mathfrak{T}_{\ell+1}&=\uppsi^{\ell+1}\big(\{01,10\}\big)\cup\left(\bigcup\limits_{p=0}^{\ell-1} \uppsi^{p+1}(00)\mathcal{T}_{\ell-p}\right)\cup \uppsi^{0}(00)\mathcal{T}_{\ell+1}\cup\left(\bigcup\limits_{p=0}^\ell \uppsi^{p+1}(11)\mathcal{T}_{\ell-p}\right)\\
&\hphantom{=}\ \cup \uppsi^{0}(11)\mathcal{T}_{\ell+1}\\
&=\uppsi^{\ell+1}\big(\{01,10\}\big)\cup\left(\bigcup\limits_{p=0}^{\ell }\uppsi^{p}(00)\mathcal{T}_{\ell-p+1}\right)\cup\left(\bigcup\limits_{p=0}^{\ell} \uppsi^{p}(11)\mathcal{T}_{\ell-p+1}\right).
\end{align*} 
\end{proof}

We might also take advantage of the definition of the map $ \uppsi $ as follows.
\begin{proof}[Alternative proof of {\hyperref[Prop:psimap]{Proposition~\ref*{Prop:psimap}}}]
For $ p \in \mathbb{N}^*$, the map $ \uppsi  $ have the following properties:

\begin{equation}\label{eq:psip}
\uppsi^p(01)=011^p,\ \ \uppsi^p(01)=101^p,\ \ \uppsi^p(00)=011^{p-1}0\ \ \mbox{and}\ \ \uppsi^p(11)=101^{p-1}0.
\end{equation}

Now recall the recurrence which link $ \mathcal{F}_n $ with $ \mathcal{T}_{n-1} $
\begin{equation*}
\mathcal{F}_0=\left\{\varepsilon\right\},\ \mathcal{F}_1=\left\{1\right\},\ \mathcal{F}_{n}=0\mathcal{T}_{n-1}\cup1\mathcal{F}_{n-1},\ n\geq2.
\end{equation*} 
We can unfold this recurrence and rewrite $ \mathcal{F}_{n} $ as belows
\begin{equation*}
\mathcal{F}_{n}=\left\{1^{n}\right\}\cup\left(\bigcup\limits_{p=0}^{n-2}1^{p}0\mathcal{T}_{n-p-1}\right),\ n\geq 2 .
\end{equation*}
Now, back to the decomposition of $ \mathfrak{T}_n $, we have
\begin{align*}
\mathfrak{T}_n&= 01\mathcal{F}_n\cup 10\mathcal{F}_n\cup00\mathcal{T}_n\cup 11\mathcal{T}_n\\
&=01\left(\left\{1^{n}\right\}\cup\left(\bigcup\limits_{p=0}^{n-2}1^{p}0\mathcal{T}_{n-p-1}\right)\right)\cup10\left(\left\{1^{n}\right\}\cup\left(\bigcup\limits_{p=0}^{n-2}1^{p}0\mathcal{T}_{n-p-1}\right)\right)\cup00\mathcal{T}_n\cup 11\mathcal{T}_n\\
&=\{011^n,101^n\}\cup\left(\bigcup\limits_{p=0}^{n-2} 011^{p}0\mathcal{T}_{n-p-1}\right)\cup00\mathcal{T}_n\cup\left(\bigcup\limits_{p=0}^{n-2} 101^{p}0\mathcal{T}_{n-p-1}\right)\cup 11\mathcal{T}_n.
\end{align*}
We conclude by \eqref{eq:psip}.
\end{proof}

As similar as the previous section, we have a recurrence formula, namely \eqref{eq:psimap}, which allows us to verify that
\begin{align*}
\#\mathfrak{T}_{n}&=\#\uppsi^{n}\big(\{01,10\}\big)+\sum\limits_{p=0}^{n-1}\# \uppsi^{p}(00)\mathcal{T}_{n-p}+\sum\limits_{p=0}^{n-1}\# \uppsi^{p}(11)\mathcal{T}_{n-p}\\
&=2+2\sum\limits_{p=0}^{n-1}(n-p)\\
&=\seqnum{A014206}(n).
\end{align*}
We confirm that the set $ \mathcal{P}_{n+1} $ and $ \mathfrak{T}_n $ have the same cardinality. The following section is then devoted to finding a bijection between these sets.
\section{The bijection}\label{Sec:Bijection}
Recall that
\begin{equation*}
\mathfrak{T}_0:=\{01,10\},\ \mathfrak{T}_n:=01\mathcal{F}_n\cup 10\mathcal{F}_n\cup 00\mathcal{T}_n\cup 11\mathcal{T}_n,\ n\geq1
\end{equation*}
and 
\begin{equation*}
\mathcal{P}_1=\{0,1\},\ \mathcal{P}_2=\{00,01,10,11\},\ \mathcal{P}_n:=\left\{0^n,1^n\right\}\cup\mathcal{P}_n^{\mathtt{01}}\cup\mathcal{P}_n^{\mathtt{10}}\cup\mathcal{P}_n^{\mathtt{00}}\cup\mathcal{P}_n^{\mathtt{11}},\ n\geq 3. 
\end{equation*}
\begin{itemize}
\item When $ n=0 $, we define an obvious one-to-one map $ \upvarphi $ defined by
\[
\upvarphi :\mathcal{P}_1=\{0,1\}\longrightarrow\mathfrak{T}_0=\{01,10\}
\]
such that $ \upvarphi (0)=01 $ and $ \upvarphi (1)=10 $.

\item When $ n=1 $, we also define an obvious one-to-one map $ \upvarphi $ defined by
\[
\upvarphi :\mathcal{P}_2=\{00,01,10,11\}\longrightarrow\mathfrak{T}_1=\{000,011,101,110\}
\]
such that $ \upvarphi (00)=000 $, $ \upvarphi (01)=011 $, $ \upvarphi (10)=101 $ and $ \upvarphi (11)=110 $.
\end{itemize}

In order to give an intuitive construction for the bijection $ \upvarphi $ between $ \mathcal{P}_{n+1} $ and $ \mathfrak{T}_n $, we also recall the following formulas:
\begin{equation*}
\mathcal{P}_{n+1}=\left\{0^{n+1},1^{n+1}\right\}\cup\left(\bigcup\limits_{p=0}^{n-1}\mathcal{P}_{n-p+1}^{\mathtt{00}}0^{p}\right)\cup\left(\bigcup\limits_{p=0}^{n-1}\mathcal{P}_{n-p+1}^{\mathtt{11}}1^{p}\right),\ n\geq 2,
\end{equation*}
and
\begin{align*}
\mathfrak{T}_n=\{011^n,101^n\}\cup\left(\bigcup\limits_{p=0}^{n-1} \uppsi^{p}(00)\mathcal{T}_{n-p}\right)\cup\left(\bigcup\limits_{p=0}^{n-1} \uppsi^{p}(11)\mathcal{T}_{n-p}\right),\ n\geq 2.
\end{align*}
Since the set $ \mathcal{P}_{n+1} $ and $ \mathfrak{T}_n $ have the same cardinality, then it actually suffices to construct an injective map. Let $ \ell $, $ r $ be nonnegative integers, and let us introduce the following notations:
\begin{itemize}
\item 	$ \pi_{\ell,r} =0^r1^{\ell-r}\in\mathcal{P}_\ell^{\mathtt{00}}$ and $ \overline{\pi}_{\ell,r} =1^r0^{\ell-r}\in\mathcal{P}_\ell^{\mathtt{11}}$ where $ 1\leq r\leq \ell-1 $;
\item $ \omega_{\ell,r} =1^r01^{\ell-r-1}\in\mathcal{T}_{\ell}$ where $ 0\leq r\leq \ell -1$.
\end{itemize}

\begin{lemma}\label{Lem:bij1}
Let the map $ \upphi$ and $\overline{\upphi }$ be defined as
\begin{align*}\label{Lem:phi}
\upphi:\left\{\pi_{\ell+1,k}=0^k1^{\ell+1-k}\ | \ 1\leq k\leq \ell\right\}&\longrightarrow \left\{\omega_{\ell,k}=1^k01^{\ell-k-1}\mid 0\leq k\leq \ell-1\right\}\\
\pi_{\ell+1,r}&\longmapsto\omega_{\ell,r-1}
\end{align*}
and
\begin{align*}
\overline{\upphi }:\left\{\overline{\pi}_{\ell+1,k}=1^k0^{\ell+1-k}\ | \ 1\leq k\leq \ell\right\}&\longrightarrow \left\{\omega_{\ell,k}=1^k01^{\ell-k-1}\mid 0\leq k\leq \ell-1\right\}\\
\overline{\pi}_{\ell+1,r}&\longmapsto\omega_{\ell,r-1}.
\end{align*}
Then $ \upphi_p$ and $\overline{\upphi }_p$ are bijective maps.
\end{lemma}
\begin{proof}
We have $ 1\leq r\leq \ell$, and it suffices to browse the index $ r $ from the domain of each map to the corresponding target set.
\end{proof}
The map $ \upphi $ can be interpreted as follows. A word in $ \mathcal{P}_n^{\mathtt{00}} $ is a compound of one block of $ 0 $'s and one block of $ 1 $'s. In the block of $ 1 $'s, replace the $ 0 $'s by $ 1 $'s except the rightmost bit, and in the block of $ 1 $'s, rightmost bit. For instance 
\begin{align*}
\upphi(\underline{0}11111111\textcolor{red}{1})&=\underline{0}11111111;\\
\upphi(\mathbf{0000}\underline{0}1111\textcolor{red}{1})&=\mathbf{1111}\underline{0}1111;\\
\upphi(\mathbf{00000000}\underline{0}\textcolor{red}{1})&=\mathbf{11111111}\underline{0}.
\end{align*}
On the other hand, we interpret the map $\overline{\upphi } $ as follows. A word in $ \mathcal{P}_n^{\mathtt{11}} $ is a compound of one block of $ 1$'s and one block of $ 0$'s. In the block of $ 1 $'s, remove the leftmost bit, and in the block of $ 0 $'s, replace the $ 0 $'s by $ 1 $'s except the leftmost bit. For instance
\begin{align*}
\overline{\upphi}(\textcolor{red}{1}\underline{0}\mathbf{00000000})&=\underline{0}\mathbf{11111111};\\
\overline{\upphi}(\textcolor{red}{1}1111\underline{0}\mathbf{0000})&=1111\underline{0}\mathbf{1111};\\
\overline{\upphi}(\textcolor{red}{1}11111111\underline{0})&=11111111\underline{0}.
\end{align*}
Now, a word in $ \mathcal{T}_n $ is a compound of one block of $ 1 $'s, one block of $ 0 $ and another block of $ 1 $'s. Only one of the block of $ 1 $'s might eventually be empty. Therefore, the inverse maps $ \upphi^{-1}$ and $\overline{\upphi }^{-1}$ are respectively described as follows. 
\begin{enumerate}[(i)]
\item Replace the leftmost block of $ 1 $'s into a block of $ 0 $'s of the same length, then append $ 1 $ at the end of the rightmost block of $ 1 $'s. The map $ \upphi^{-1} $ is defined as
\begin{align*}
\upphi^{-1}:\left\{\omega_{\ell,k}=1^k01^{\ell-k-1}\mid 0\leq k\leq \ell-1\right\}&\longrightarrow \left\{\pi_{\ell+1,k}=0^k1^{\ell+1-k}\ | \ 1\leq k\leq \ell\right\}\\
\omega_{\ell,r}&\longmapsto\pi_{\ell+1,r+1}.
\end{align*}
\item Replace the rightmost block of $ 1 $'s into a block of $ 0 $'s of the same length, then append $ 1 $ at the beginning of the leftmost block of $ 1 $'s. The map $ \overline{\upphi}^{-1} $ is in turn defined as
\end{enumerate}
\begin{align*}
\overline{\upphi }^{-1}:\left\{\omega_{\ell,k}=1^k01^{\ell-k-1}\mid 0\leq k\leq \ell-1\right\}&\longrightarrow \left\{\overline{\pi}_{\ell+1,k}=1^k0^{\ell+1-k}\ | \ 1\leq k\leq \ell\right\}\\
\omega_{r}&\longmapsto\overline{\pi}_{\ell+1,r+1}.
\end{align*}
\begin{corollary}
Let $ p $ be a nonnegative integer with $p\in\{1,2,\ldots,\ell-2\}$, and let the map $ \upphi _{p}$ and $\overline{\upphi }_{p}$ be defined as
\begin{align*}
\upphi _{p}:\left\{\pi_{\ell-p+1,k}\ |\ 1\leq k\leq \ell-p\right\}0^p&\longrightarrow \uppsi^p(00)\left\{\omega_{\ell-p,k}\mid 0\leq k\leq \ell-p-1\right\}\\
\pi_{\ell-p+1,r}0^p&\longmapsto\uppsi^p(00)\omega_{\ell-p,r-1}
\end{align*}
and
\begin{align*}
\overline{\upphi} _{p}:\left\{\overline{\pi}_{\ell-p+1,k}\ |\ 1\leq k\leq \ell-p\right\}1^p&\longrightarrow \uppsi^p(11)\left\{\omega_{\ell-p,k}\mid 0\leq k\leq \ell-p-1\right\}\\
\overline{\pi}_{\ell-p+1,r}1^p&\longmapsto\uppsi^p(11)\omega_{\ell-p,r-1}.
\end{align*}
Then $ \upphi_{p}$ and $\overline{\upphi }_{p}$ are bijective maps.
\end{corollary}

We complete the bijection by imposing $ \upvarphi(0^{n+1})=011^n $ and $ \upvarphi(1^{n+1})=101^n $. In practice, it is more convenient to decompose the map $ \upvarphi $ into the following restrictions. First of all, the ``new regions'' mapping:
\[\begin{array}{rcr}
\begin{array}{rl}
\upvarphi _{00}:\mathcal{P}_{n+1}^{\mathtt{00}}&\longrightarrow 00\mathcal{T}_{n}\\
\pi&\longmapsto 00\upphi(\pi)
\end{array}
& \mbox{and} &
\begin{array}{rl}
\upvarphi _{11}:\mathcal{P}_{n+1}^{\mathtt{11}}&\longrightarrow 11\mathcal{T}_{n}\\
\pi&\longmapsto 11\overline{\upphi}(\pi).
\end{array}
\end{array}\]
Then, let $ \pi \in \mathcal{P}_{n+1}^{\mathtt{01}}$ (resp. $ \pi \in \mathcal{P}_{n+1}^{\mathtt{10}}$), and let $ p $ be the largest nonnegative integer, with $ p<n-1 $, such that we can write $ \pi $ as the compound $ \pi=\pi_1 0^p $ (resp. $ \pi=\pi_1 1^p $). The associated maps are
\begin{align*}
\upvarphi _{01}:\widehat{\mathcal{P}}_{n+1}^{\mathtt{01}}:=\mathcal{P}_{n+1}^{\mathtt{01}}\cup\{0^{n+1}\}&\longrightarrow 01\mathcal{F}_{n}\\
\pi&\longmapsto\begin{cases}
011^{n}, & \text{if $\pi=0^{n+1}$;}\\
011^{p-1}0\upphi(\pi_1), & \text{if $ \pi=\pi_1 0^p$}
\end{cases}
\end{align*}
and
\begin{align*}
\upvarphi _{10}:\widehat{\mathcal{P}}_{n+1}^{\mathtt{10}}:=\mathcal{P}_{n+1}^{\mathtt{10}}\cup\{1^{n+1}\}&\longrightarrow 10\mathcal{F}_{n}\\
\pi&\longmapsto\begin{cases}
101^{n}, & \text{if $\pi=1^{n+1}$;}\\
101^{p-1}0\overline{\upphi}(\pi_1), & \text{if $ \pi=\pi_1 1^p$}.
\end{cases}
\end{align*}
The inverse image $ \upvarphi^{-1} $ can be constructed as belows.
\begin{itemize}
\item If $ \omega=00\omega_1\in00\mathcal{T}_n $ then $ \upvarphi_{00}^{-1}(\omega)=\upphi^{-1}(\omega_1) $;
\item if $ \omega=11\omega_1\in11\mathcal{T}_n $ then $ \upvarphi_{00}^{-1}(\omega)=\overline{\upphi}^{-1}(\omega_1) $;
\item if $ \omega=\omega_0\omega_1\in01\mathcal{F}_n $ where $ \omega_1 \in\mathcal{T}_\ell$ for some nonnegative integer $ \ell $, then 
$$ \upvarphi_{01}^{-1}(\omega)=\upphi^{-1}(\omega_1) 0^{n-\ell}; $$
\item if $ \omega=\omega_0\omega_1\in10\mathcal{F}_n $ where $ \omega_1 \in\mathcal{T}_\ell$ for some nonnegative integer $ \ell $, then 
$$ \upvarphi_{10}^{-1}(\omega)=\overline{\upphi}^{-1}(\omega_1) 1^{n-\ell} .$$
\end{itemize}
Note that the key ingredient of the bijection is to identify a prefix for the one, and then transform it into a suffix for the other. Same scheme for the inverse. The following examples aims at illustrating these formulas.
\begin{enumerate}
\item Let $ 0001111111\in\mathcal{P}_{10} $. Then 
\begin{align*}
\upvarphi_{00}(0001111111)&=\textcolor{blue}{00}\upphi(\mathbf{00}\underline{0}111111\textcolor{red}{1}) \\
&=\textcolor{blue}{00}\mathbf{11}\underline{0}111111 \in00\mathcal{T}_{9}.
\end{align*}
\item Let $ 1110111111\in\mathfrak{T}_{8}$. Since $ 1110111111\in11\mathcal{T}_{8}$, we have 
\begin{align*}
\upvarphi^{-1}_{11}(\textcolor{blue}{11}10111111)&=\overline{\upphi}^{-1}(10\mathbf{111111})\\
&=\textcolor{red}{1}10\mathbf{000000}\in\mathcal{P}_9^{\mathtt{11}}.
\end{align*}
\item Let $ 1111000\textcolor{blue}{11111}\in\mathcal{P}_{12} $. We identify the suffix $ \textcolor{blue}{11111} $, whose associated prefix is $ \textcolor{blue}{1011110}$, and the factor $ 1111000 $ which is a word in $ \mathcal{P}_7^{\mathtt{11}} $. Therefore we have $ \overline{\upphi}(\textcolor{red}{1}111\underline{0}\mathbf{00}) = 111\underline{0}\mathbf{11}$. Finally we obtain the word $\textcolor{blue}{1011110}111\underline{0}\mathbf{11}\in10\mathcal{F}_{11} $.
\item Let $\textcolor{blue}{011111110}111011111 \in \mathfrak{T}_{16}$. The prefix and the corresponding suffix are respectively $\textcolor{blue}{011111110} $ and $ \textcolor{blue}{0000000} $, i.e, the remaining factor has to be mapped to $ \mathcal{P}_{10}^{\mathtt{01}} $. Now, since the factor $ 111011111 $ is a word in $ \mathcal{T}_9 $, the associated transformation is $ \upphi(\mathbf{111}011111) =\mathbf{000}011111\textcolor{red}{1}$. So, the final encoding is $ \mathbf{000}011111\textcolor{red}{1}\textcolor{blue}{0000000} \in\mathcal{P}_{17}$.
\end{enumerate}

Finally, we perform the bijection entrywise in \Tabs{tab:bijection} for $ n=0,1,2,3,4,5 $. For instance, we read $ \upvarphi _{00}(0011)=0010 $, $ \upvarphi _{11}(11100)=111101 $, 
$\upvarphi _{01}(011110)=0100111 $ and $\upvarphi _{10}(101)=1000 $.

 When $ n>0 $, then the entries in
\begin{itemize}
\item $ 01\mathcal{F}_{n} $ and $ 10\mathcal{F}_{n} $ are obtained by respectively applying $ \uppsi  $ to the entries of the previous top cell, i.e,
\begin{equation*}
 01\mathcal{F}_n:=\uppsi\left(01\mathcal{F}_{n-1}\cup00\mathcal{T}_{n-1}\right)\ \mbox{and} \  10\mathcal{F}_n:=\uppsi\left(10\mathcal{F}_{n-1}\cup11\mathcal{T}_{n-1}\right); 
\end{equation*}
\item $ \widehat{\mathcal{P}}_{n+1}^{\mathtt{01}} $ and $ \widehat{\mathcal{P}}_{n+1}^{\mathtt{10}} $ are obtained by respectively appending $ 0 $ and $ 1 $ at the right-end of the entries of the previous top cell, i.e.,
\begin{equation*}
\widehat{\mathcal{P}}_{n}^{\mathtt{01}} :=\left(\widehat{\mathcal{P}}_{n-1}^{\mathtt{01}} \cup\mathcal{P}_{n-1}^{\mathtt{00}} \right)0\ \mbox{and} \  \widehat{\mathcal{P}}_{n}^{\mathtt{10}} :=\left(\widehat{\mathcal{P}}_{n-1}^{\mathtt{10}} \cup\mathcal{P}_{n-1}^{\mathtt{11}} \right)1.
\end{equation*}
\end{itemize}
For example, consider the case $ n=5 $.

\begin{table}[H]
\centering
{
\def\arraystretch{1.25}
\begin{tabular}{|c|ll|rl||ll|rl|}
 
\hline
$n$ & $\widehat{\mathcal{P}}_{n+1}^{\mathtt{01}}$ & $\mathcal{P}_{n+1}^{\mathtt{00}}$ & $01\mathcal{F}_{n}$ & $00\mathcal{T}_{n}$ & $\widehat{\mathcal{P}}_{n+1}^{\mathtt{10}}$ & $\mathcal{P}_{n+1}^{\mathtt{11}}$ & $10\mathcal{F}_{n}$ & $11\mathcal{T}_{n}$\\
\hline 
\hline
$0$ & $0$ & & $01$ & & $1$ & & $10$ & \\
\hline 
$1$ & $00$ & $01$ & $011$ & $000$ & $11$ & $10$ & $101$ & $110$\\
\hline 
\multirow{2}{*}{$2$} & $000$ & $011$ & $0111$ & $0001$ & $111$ & $100$ & $1011$ & $1101$\\
 & $010$ & $001$ & $0100$ & $0010$ & \framebox[1.1\width]{$101$} & $110$ & \framebox[1.1\width]{$1000$} & $1110$\\
\hline 
\multirow{4}{*}{$3$} & $0000$ & $0111$ & $01111$ & $00011$ & $1111$ & $1000$ & $10111$ & $11011$\\
 & $0100$ & \framebox[1.1\width]{$0011$} & $01100$ & \framebox[1.1\width]{$00101$}& $1011$ & $1100$ & $10100$ & $11101$\\
 & $0110$ & $0001$ & $01001$ & $00110$ & $1001$ & $1110$ & $10001$ & $11110$\\
 & $0010$ & & $01010$ & & $1101$ & & $10010$ & \\
\hline 
\multirow{7}{*}{$4$} & $00000$ & $01111$ & $011111$ & $000111$ & $11111$ & $ 10000$ & $101111$ & $110111$\\
 & $01000$ & $00111$ & $011100$ & $001011$ & $10111$ & $ 11000$ & $101100$ & $111011$\\
 & $01100$ & $00011$& $011001$ & $001101$ & $10011$ & \framebox[1.1\width]{$ 11100$}& $101001$ & \framebox[1.1\width]{$111101$}\\
 & $00100$ &$ 00001$& $011010$ & $001110$ & $11011$ & $ 11110$ & $101010$ & $111110$\\
 & $01110$ & & $010011$ & & $10001$ & & $100011$ & \\
 & $00110$ & & $010101$ & & $11001$ & & $100101$ & \\
 & $00010$ & & $010110$ & & $11101$ & & $100110$ & \\
\hline 
\multirow{11}{*}{$5$} & $00000\textcolor{red}{0}$ & $011111$ & $\textcolor{red}{011}1111$ & $0001111$ & $11111\textcolor{red}{1}$ & $100000$ & $\textcolor{red}{101}1111$ & $1101111$\\
 & $01000\textcolor{red}{0}$ & $001111$ & $\textcolor{red}{011}1100$ & $0010111$ & $10111\textcolor{red}{1}$ & $110000$ & $\textcolor{red}{101}1100$ & $1110111$\\
 & $01100\textcolor{red}{0}$ & $000111$ & $\textcolor{red}{011}1001$ & $0011011$ & $10011\textcolor{red}{1}$ & $111000$ & $\textcolor{red}{101}1001$ & $1111011$\\
 & $00100\textcolor{red}{0}$ & $000011$ & $\textcolor{red}{011}1010$ & $0011101$ & $11011\textcolor{red}{1}$ & $111100$ & $\textcolor{red}{101}1010$ & $1111101$\\
 & $01110\textcolor{red}{0}$ & $000001$ & $\textcolor{red}{011}0011$ & $0011110$ & $10001\textcolor{red}{1}$ & $111110$ & $\textcolor{red}{101}0011$ & $1111110$\\
 & $00110\textcolor{red}{0}$ & & $\textcolor{red}{011}0101$ & & $11001\textcolor{red}{1}$ & & $\textcolor{red}{101}0101$ & \\
 & $00010\textcolor{red}{0}$ & & $\textcolor{red}{011}0110$ & & $11101\textcolor{red}{1}$ & & $\textcolor{red}{101}0110$ & \\
 &\framebox[1.1\width]{$ 01111\textcolor{red}{0}$} & & \framebox[1.1\width]{$\textcolor{red}{010}0111$} & & $ 10000\textcolor{red}{1}$ & & $\textcolor{red}{100}0111$ & \\
 &$ 00111\textcolor{red}{0}$ & & $\textcolor{red}{010}1011$ & & $ 11000\textcolor{red}{1}$ & & $\textcolor{red}{100}1011$ & \\
 & $00011\textcolor{red}{0}$ & & $\textcolor{red}{010}1101$ & & $ 11100\textcolor{red}{1}$ & & $\textcolor{red}{100}1101$ & \\
 & $00001\textcolor{red}{0}$ & & $\textcolor{red}{010}1110$ & & $ 11110\textcolor{red}{1}$ & & $\textcolor{red}{100}1110$ & \\
\hline 
\end{tabular}
}
\caption{The sets $ \mathfrak{T}_n $ and $ \mathcal{P}_{n+1}$, $ n=0,1,2,3,4,5 $.}
\label{tab:bijection}
\end{table}


\begin{thebibliography}{10}

\bibitem{Adams}
C.~Adams. \newblock{\em The Knot Book}.\newblock W. H. Freeman and Company, 1994.

\bibitem{APS}
N.~Alon, H.~Last, R.~Pinchasi and M. Sharir. \newblock On the complexity of arrangements of circles in the plane. \newblock {\em Discrete Comput. Geom.}, 26(4):465--492, 2001.

\bibitem{Batcher}
K.~E. Batcher. \newblock Sorting networks and their applications.\newblock In {\em Proc. AFIPS Spring Joint Comput. Conf.}, volume 32, 1968, pages 307--314. ACM, 1968.

\bibitem{CR}
R.~Cluzel and J.-P.~Robert. \newblock {\em La G\'eom\'etrie et ses Applications}. \newblock Delagrave, 1971.

\bibitem{DD}
D.~Denton and P.~Doyle. \newblock Shadow movies not arising from knots. \newblock \arxiv{1106.3545}, 2011.

\bibitem{GKP}
R.~L. Graham, D.~E. Knuth and O.~Patashnik. \newblock{\em Concrete Mathematics: A Foundation	for Computer Science}. \newblock Addison-Wesley, 1994.

\bibitem{HK} 
F.~Harary and Louis~H. Kauffman. \newblock Knots and graphs I--Arc graphs and colorings. \newblock {\em Advances in Applied Mathematics}, 22(3):312--337, 1999.

\bibitem{JR}
S.~Jablan and L.~Radovic. \newblock Knot polynomials: myths and reality. \newblock\arxiv{1107.1877}, 2011.

\bibitem{JH}
I.~Johnson and A.~K. Henrich. \newblock {\em An Interactive Introduction to Knot Theory}. \newblock Dover Publications, Inc, 2017.

\bibitem{Kauffman1}
Louis~H. Kauffman. \newblock {\em Knots and Physics}. \newblock World Scientific Publishers, 1991.

\bibitem{Kauffman} 
Louis~H. Kauffman. \newblock State models and the Jones polynomial. \newblock {\em Topology}, 26(3):395--107, 1987.

\bibitem{Lang} 
H.~W.~Lang. \newblock Bitonic sort. \newblock\url{http://www.iti.fh-flensburg.de/lang/algorithmen/sortieren/bitonic/bitonicen.htm}, 2016 [Accessed 10 Dec. 2017].

\bibitem{Manturov}
V.~Manturov. \newblock {\em Knot Theory}. \newblock CRC Press, 2004.

\bibitem{RR}
F.~Ramaharo and F.~Rakotondrajao. \newblock A state enumeration of the foil knot. \newblock \arxiv{1712.04026}, 2017.

\bibitem{Rosin} 
P.~Rosin. \newblock Rosettes and other arrangements of circles. \newblock {\em Nexus Network Journal}, 3(2):113--126, 2001.

\bibitem{Sloane} 
N.~J.~A. Sloane. \newblock The On-Line Encyclopedia of Integer Sequences. \newblock Published electronically at \url{http://oeis.org}, 2017.

\end{thebibliography}
\end{document}